\documentclass{siamart190516}

\usepackage{hyperref}
\usepackage{amsmath}
\usepackage[english]{babel}
\usepackage[utf8]{inputenc}
\usepackage{amsfonts}
\usepackage{color}
\usepackage{geometry}
\newsiamremark{rmk}{Remark}
\crefname{rmk}{Remark}{Remarks}

\DeclareMathOperator{\Corr}{Corr}
\DeclareMathOperator{\BF}{\mathcal{BF}}

\DeclareMathOperator{\E}{\mathbb{E}}

\DeclareMathOperator{\N}{\mathbb{N}}

\DeclareMathOperator{\R}{\mathbb{R}}
\DeclareMathOperator{\cF}{\mathcal{F}}
\DeclareMathOperator{\cG}{\mathcal{G}}

\DeclareMathOperator{\cD}{\mathcal{D}}

\DeclareMathOperator{\cI}{\mathcal{I}}
\DeclareMathOperator{\cB}{\mathcal{B}}

\DeclareMathOperator{\fe}{\mathfrak{e}}
\DeclareMathOperator{\fs}{\mathfrak{s}}

\DeclareMathOperator{\bP}{\mathbb{P}}
\DeclareMathOperator{\bH}{\mathbb{H}}
\DeclareMathOperator{\C}{\mathbb{C}}

\DeclareMathOperator{\mm}{\mathbf{m}}

\newcommand{\pd}[2]{\frac{\partial #1}{\partial #2}}

\newcommand{\der}[2]{\frac{d #1}{d #2}}
\newcommand{\dersup}[3]{\frac{d^{#3} #1}{d #2^{#3}}}

\newcommand{\Norm}[2]{\left\Vert #1 \right\Vert_{#2}}

\headers{Non-Local Pearson diffusions}{G. Ascione, N. Leonenko and E. Pirozzi}

\title{Non-Local Pearson diffusions}

\author{Giacomo Ascione\thanks{Dipartimento di Matematica e Applicazioni ``Renato Caccioppoli'', Università degli Studi di Napoli Federico II, 80126 Napoli, Italy 
  (\email{giacomo.ascione@unina.it},\email{enrica.pirozzi@unina.it}).}
\and Nikolai Leonenko\thanks{School of Mathematics, Cardiff University, Cardiff CF24 4AG, UK 
  (\email{leonenkon@cardiff.ac.uk}).}
\and Enrica Pirozzi\footnotemark[1]}

\begin{document}

\maketitle

\begin{abstract}
In this paper we focus on strong solutions of some heat-like problems with a non-local derivative in time induced by a Bernstein function and an elliptic operator given by the generator or the Fokker-Planck operator of a Pearson diffusion. Such kind of non-local equations naturally arise in the treatment of particle motion in heterogeneous media. In particular, we use spectral decomposition results for the usual Pearson diffusion to exploit explicit solutions of the aforementioned equations. Moreover, we provide stochastic representation of such solutions in terms of time-changed Pearson diffusions. Finally, we exploit some further properties of these processes, such as limit distributions and long/short-range dependence.
\end{abstract}





\section{Introduction}
Pearson distributions \cite{pearson1914tables} constitute a family of probability distributions whose density functions $m(x)$ satisfy the so called Pearson equation
\begin{equation*}
\frac{m'(x)}{m(x)}=\frac{b_0+b_1x}{d_0+d_1x+d_2x^2}
\end{equation*} 
as $x \in E \subseteq \R$. As one defines the polynomial $\mu(x)=a_0+a_1x$ and $D(x)=d_0+d_1x+d_2x^2$ and sets $\sigma(x)=\sqrt{2D(x)}$, Pearson equation is satisfied by the stationary measure of the solutions of the Stochastic Differential Equation (SDE)
\begin{equation*}
dX(t)=\mu(X(t))dt+\sigma(X(t))dW(t),
\end{equation*}
where $W(t)$ is a standard Brownian motion. For this reason, such kind of diffusions, called Pearson diffusions \cite{forman2008pearson}, are statistically tractable with many different tools. For instance, they naturally arise while studying non-linear time series linked to dynamical models with random perturbations (see, for instance, \cite{ozaki19852}). An important tool to deal with Pearson diffusions is the spectral decomposition, proposed in \cite{linetsky2007spectral}. Indeed Pearson diffusions can be subdivided in three spectral categories depending on the spectrum of the generator. In the first spectral category, we have Pearson diffusions whose generator admit purely discrete spectrum and then a spectral decomposition of the transition density of the process follows easily. In the second spectral category, we have Pearson diffusions whose generator admits a discrete spectrum and an absolutely continuous spectrum (both simple) separated by a cut-off value. Such Pearson diffusions are more difficult to study due to the fact that the eigenfunctions in the continuous spectrum are expressed in terms of hypergeometric functions, as shown in \cite{leonenko2010statistical} and \cite{avram2013spectralfs}. Finally, in the third spectral category we only have Student diffusions, whose generator admits a simple discrete spectrum and an absolutely continuous spectrum of multiplicity two separated by a cut-off value. This case has been studied in \cite{leonenko2010statisticalb, arista2020explicit}. Second and third spectral category go under the name of heavy tailed Pearson diffusions and their spectral properties have been analysed in \cite{avram2013spectral}.\\
In the modern theory of diffusion processes, anomalous diffusions have shown to be quite useful (see \cite{metzler2000random,henry2010introduction}) and stochastic models for such diffusions have been widely studied (see \cite{scalas2003revisiting,gorenflo2003fractional,meerschaert2011stochastic}). Anomalous diffusions models naturally arise when studying the motion of particles in heterogeneous media. For such kind of models, different analytic techniques are known (see, for instance, \cite{bouchaud1990anomalous}). Such models fall into the class of fractional motions, which are described via a Langevin-type equation (see \cite{eliazar2013fractional}). A class of anomalous diffusions can be seen as limit of continuous time random walks whose inter-jump times have infinite mean. Such limits lead to the appearance of the fractional Caputo derivative in the Fokker-Planck equation (see \cite{metzler2000random}). Despite Laplace transforms methods let us obtain the solution up to inverse Laplace transform, fractional Fokker-Planck equations are not usually solved explicitly. In this context, the spectral decomposition of Pearson diffusions has been revealed to be a quite powerful tool to explicitly express strong solutions of fractional diffusion equations.
This lead to the introduction of fractional Pearson diffusions: in \cite{leonenko2013fractional} the first spectral category was covered, in \cite{leonenko2017heavy} the authors focused on the second spectral category, while, up to our knowledge, there were no known methods to extend the spectral decomposition to the fractional case in the third spectral category. In the second case, to express strong solutions of fractional Kolmogorov equations, the semigroup approach presented in \cite{baeumer2001stochastic} has been used. In both cases, the stochastic representation of such solutions is given by means of time-changed Pearson diffusions (with inverse stable subordinators). Similar strategies have been shown to work for lattice approximation of fractional Pearson diffusions: in \cite{ascione2019fractional} the spectral decomposition of a fractional immigration-death process (that is the lattice approximation of the Ornstein-Uhlenbeck process) is presented. On the other hand, subordinated Pearson diffusions (in particular the subordinated Jacobi process) have been shown to be useful tools to obtain large deviation principles (see \cite{demni2009large}). Let us also remark that spectral properties of time-changed Markov processes can be used to determine the behaviour of their correlation structure, as shown in \cite{leonenko2013correlation,patie2019spectral}.\\
However, the classical fractional model is not the unique way one can achieve an anomalous diffusion. Indeed, anomalous diffusions can be introduced also by considering more general relaxation patterns than the exponential or the Mittag-Leffler ones (see \cite{oliveira2019anomalous}). Different inter-jump times in continuous time random walks and different relaxation patterns lead to a wider class of anomalous diffusions obtained via subordination of a random process (see \cite{sokolov2005diffusion}). In this case, one has to consider a more general non-local derivative in place of the Caputo one in Fokker-Planck equations.
In \cite{kochubei2011general} and \cite{toaldo2015convolution} other non-local derivatives have been constructed, with the aid of Bernstein functions \cite{schilling2012bernstein}, that are the Laplace exponents of subordinators. A non-local derivative that can be defined in this way is, for instance, the tempered fractional derivative, which is an alternative to the classical fractional derivative that preserve the finiteness of the mean of the associated subordinator. The tempered model has been widely used to describe anomalous diffusions in heterogeneous media (in particular in geophysics, as in \cite{meerschaert2014tempered,zhang2012linking}) and its properties are well-known (see, for instance, \cite{alrawashdeh2017applications}). A wide discussion on relaxation patterns can be found in \cite{meerschaert2019relaxation}. Generalized fractional calculus is strictly linked with the definition of time-changed Markov processes. Indeed, for instance, in \cite{toaldo2015convolution,chen2017time} a link between abstract generalized fractional differential equations and time-changed semigroups is established, while in \cite{da2020green} properties of the Green measures of time-changed Markov properties are exploited. Moreover, such non-local derivatives have been used for instance to define a class of non-local birth-death processes (see \cite{ascione2020non}) whose stationary distributions fall into  the Katz family, which are discrete analogous of Pearson diffusions of the first spectral category. On the other hand, a first generalization to the general non-local setting of a Pearson diffusion (in particular the Ornstein-Uhlenbeck process) has been achieved in \cite{gajda2015time}. Let us also recall that in \cite{demni2009large} the authors study a subordinated Jacobi process (with an inverse Gaussian subordinator), thus leading to a time-changed version of the aforementioned process.\\

The aim of the paper is to give a unified framework for working with non-local generalizations of Pearson diffusions and non-local equations that are linked to them, exhibiting exact strong solutions for this particular class of non-local advection-diffusion equations. In particular, we focus on non-local Cauchy problems of the form
\begin{equation}\label{forintro}
\begin{cases}
\partial_t^\Phi u(t,y)=\cG u(t,y) & t>0, \ y \in E \\
u(0,y)=g(y) & y \in E
\end{cases}
\end{equation}
and
\begin{equation}\label{backintro}
\begin{cases}
\partial_t^\Phi v(t,x)=\cF v(t,x) & t>0, \ x \in E \\
v(0,x)=f(x) & x \in E
\end{cases}
\end{equation}
where $\cG$ and $\cF$ are respectively the generator and the Fokker-Planck operator of a Pearson diffusion and $\partial_t^\Phi$ is a Caputo-type non-local derivative linked to a Bernstein function $\Phi$. We use the notion of time-changed process to introduce the family of non-local Pearson diffusions and then provide the spectral decomposition of the transition densities of such diffusions. Moreover, we use both the spectral decomposition and a semigroup approach to exploit strong solutions of the aforementioned non-local Cauchy problems. Let us remark that in this paper we cover also the case of the third spectral category, for which spectral results for non-local generalizations were (up to our knowledge) unknown even in the standard fractional case.\\

\Cref{Sec2} is dedicated to a complete recap on spectral properties of classical Pearson diffusions, while \Cref{Sec3} presents some basic properties of Bernstein functions, inverse subordinators and non-local derivatives. We also specify the semigroup theory results given in \cite{baeumer2001stochastic,toaldo2015convolution,chen2017time} to the case of complete Bernstein functions, in which better regularity can be proven. The proof of such result is quite technical, thus it is left in \cref{AppA}. The definition and some preliminary properties of non-local Pearson diffusions are presented in \Cref{Sec4}, with particular attention to the existence of a transition probability density. In \Cref{Sec5} we focus on non-local Pearson diffusions of the first spectral category, extending the results given in \cite{leonenko2013fractional} to general Bernstein functions. In \Cref{Sec6} we consider non-local Pearson diffusions of the second (adapting the approach given in \cite{leonenko2017heavy}) and the third spectral categories. \Cref{Sec7} is devoted to stochastic representation results of strong solutions of \cref{forintro,backintro}. Finally, in \Cref{Sec8} we explore some further properties of the non-local Pearson diffusions, as limit distributions, first-order stationarity and long/short-range dependence.
\section{Pearson diffusions and their spectral classification}\label{Sec2}
From now on let us fix a filtered probability space $(\Omega, \Sigma, \cF_t, \bP)$. Let us give the definition of Pearson diffusion as presented in \cite{forman2008pearson}.
\begin{definition}
	A \textbf{Pearson diffusion} $X(t)$ is a diffusion process satisfying the following stochastic differential equation
	\begin{equation}\label{PSDE}
	dX(t)=\mu(X(t))dt+\sigma(X(t))dW(t),
	\end{equation}
	where $W(t)$ is a standard Brownian motion and $\mu(x)$ and $\sigma^2(x)$ are polynomials respectively of at most first and second degree.
\end{definition}
In particular let us set
\begin{align*}
\mu(x)=a_0+a_1x && D(x)=\frac{\sigma^2(x)}{2}=d_0+d_1x+d_2x^2.
\end{align*}
In particular the diffusion space of $X(t)$ is given by an interval $E=(l,L)$ in which the polynomial $D(x)$ is positive (hence the function $\sigma(x)$ is real and non-zero).\\
Moreover, let us recall that the family of operators $(T(t))_{t \ge 0}$ acting on $C_0(E)$, i.e. the space of continuous functions on the closure $\bar{E}$ of $E$ that are $0$ at infinity (if $E$ is bounded then $C_0(E)=C(\bar{E})$), given by
\begin{equation*}
T(t)f(x)=\E_x[f(X(t))],
\end{equation*}
where $\bP_x(\cdot)=\bP(\cdot|X_0=x)$, is a uniformly bounded strongly continuous $C_0$-semigroup on $L^2(E)$ (this is proved, for instance, in \cite{leonenko2013fractional,leonenko2017heavy}). In particular its generator $\cG$ is defined as
\begin{equation*}
\cG g(y)=\left[\mu(y)\der{}{y}+D(y)\dersup{}{y}{2}\right]g(y),
\end{equation*}
with operator core $C^2_b(E)$. \\
Concerning the process $X(t)$, the transition probability density $p(t,x;y)$ is well defined for $x,y \in E$. In particular, it is solution of the following Cauchy problem (the \textit{backward problem})
\begin{equation}\label{classCaupback}
\begin{cases}
\pd{p}{t}(t,x;y)=\cG p(t,x;y) & t>0, \ y \in E\\
p(0,x;y)=\delta_x(y) & y \in E
\end{cases}
\end{equation}
where $\delta_x$ is a Dirac delta centred in $x \in E$.\\
On the other hand, we can also define the Fokker-Planck operator, or just forward operator, as
\begin{equation*}
\cF f(x)=-\der{}{x}(\mu(x)f(x))+\dersup{}{x}{2}(D(x)f(x)),
\end{equation*}
with operator core $C^2(E)$. In particular $p(t,x;y)$ is also solution of the Cauchy problem (the \textit{forward problem}):
\begin{equation}\label{classCaupfor}
\begin{cases}
\pd{p}{t}(t,x;y)=\cF p(t,x;y) & t>0, \ x \in E\\
p(0,x;y)=\delta_y(x) & x \in E.
\end{cases}
\end{equation}
In particular, if the process $X(t)$ admits a stationary measure $\mm$, then its density $m(x)$ must satisfy a stationary version of the forward equation, that becomes:
\begin{equation}\label{Peareq}
\frac{m'(x)}{m(x)}=\frac{(a_0-d_1)+(a_1-2d_2)x}{d_0+d_1x+d_2x^2}.
\end{equation}
Such equation is called \textbf{Pearson equation}, by the fact that it was introduced in \cite{pearson1914tables} to classify some important classes of distributions. From now on we will only work with Pearson diffusions that admit a stationary measure that is also the limit measure of the process. In particular this means that, up to a re-parametrization, $\mu(x)=-b_0(x-b_1)$ for some $b_0>0$ and $b_1 \in \R$.\\
After such observation, we can recognize six different Pearson diffusions depending on the coefficients of $D$:
\begin{itemize}
	\item If $D\equiv d_0$, then $X(t)$ is a Ornstein-Uhlenbeck (OU) process and the stationary distribution is a Gaussian distribution;
	\item If $D(x)=d_0+d_1x$, then $X(t)$ is a Cox-Ingersoll-Ross (CIR) process and the stationary distribution is a Gamma distribution;
	\item If $D(x)=d_0+d_1x+d_2x^2$ with $d_2<0$, then $X(t)$ is a Jacobi process and the stationary distribution is a Beta distribution;
	\item If $D(x)=d_0+d_1x+d_2x^2$ with $d_2>0$ and the discriminant $\Delta_D>0$, then $X(t)$ is a Fisher-Snedecor (FS) process and the stationary distribution is a Fisher-Snedecor  distribution;
	\item If $D(x)=d_0+d_1x+d_2x^2$ with $d_2>0$ and the discriminant $\Delta_D=0$, then $X(t)$ is a reciprocal Gamma (RG) process and the stationary distribution is a reciprocal Gamma distribution;
	\item If $D(x)=d_0+d_1x+d_2x^2$ with $d_2>0$ and the discriminant $\Delta_D<0$, then $X(t)$ is a Student process and the stationary distribution is a Student distribution.
\end{itemize}
Depending on the property of the spectrum of the generator $\cG$, in \cite{leonenko2017heavy} these distributions were subdivided in three categories depending on the spectral category of Linetski classification (see \cite[Theorem $3.2$]{linetsky2007spectral}):
\begin{itemize}
	\item The first spectral category contains the OU, the CIR and the Jacobi processes: their generator $\cG$ admits purely discrete spectrum with infinitely many single non-positive eigenvalues $(\lambda_n)_{n \in \N}$;
	\item The second spectral category contains the FS and the RG processes: their generator $\cG$ admits a discrete part and an absolutely continuous part that are disjoint and both of multiplicity one;
	\item The third spectral category contains only the Student processes: its generator $\cG$ admits a discrete part of multiplicity one and a disjoint absolutely continuous part of multiplicity two.
\end{itemize}
Such classification is based on the oscillatory/non-oscillatory behaviour of the endpoints of the Sturm-Liouville equations $\cG f=-\lambda f$ (see \cite{dunford1958linear,weidmann2006spectral,amrein2005sturm}).\\
Let us give some details on each Pearson diffusion. In particular we will re-parametrize again the polynomials $D(x)$ and $\mu(x)$ in a form that will make the writing of the parameters of the stationary distributions easier. Moreover, we define the \textit{normalized polynomials}, where the normalization constant is chosen with respect to the stationary density $m(x)$, which is also the orthogonality density of the polynomials. In the following we will denote by $Q_n$ the normalized polynomials and with $K_n$ the normalization constants.
\subsection{Pearson diffusions of spectral category I}
\subsubsection{The OU process}
The OU process is solution of the SDE
\begin{equation*}
dX(t)=-\theta(X(t)-\mu)dt+\sqrt{2\theta \sigma^2}dW(t), \ t \ge 0
\end{equation*}
as $\theta>0$ and $\mu,\sigma \in \R$. The diffusion space is given by $E=\R$ and its stationary density is a Gaussian one, given by
\begin{equation*}
m(x)=\frac{1}{\sqrt{2\pi \sigma^2}}e^{-\frac{(x-\mu)^2}{2\sigma^2}}, \ x \in \R 
\end{equation*}
Concerning the eigenvalue equation $\cG f=-\lambda f$, it admits solutions for $\lambda_n=\theta n$ as $n \ge 0$ and its solutions are the Hermite polynomials (see \cite{schoutens2012stochastic}), defined by means of the Rodrigues formula (see \cite{ismail2005classical})
\begin{equation*}
H_n(x)=(-1)^n(m(x))^{-1}\dersup{}{x}{n}m(x), \ x \in \R, \ n \in \N_0
\end{equation*}
with normalization constants $K_n=\frac{\sigma^n}{\sqrt{n!}}$.
\subsubsection{The CIR process}
The CIR process is solution of the SDE
\begin{equation*}
dX(t)=-\theta\left(X(t)-\frac{b}{a}\right)dt+\sqrt{\frac{2\theta}{a}X(t)}dW(t), \ t \ge 0
\end{equation*}
where $\theta,a,b>0$. The diffusion space is given by $E=(0,+\infty)$ and its stationary density is the Gamma one, given by
\begin{equation*}
m(x)=\frac{a^b}{\Gamma(b)}x^{b-1}e^{-ax}, \ x>0.
\end{equation*}
Even in this case, the eigenvalue equation $\cG f=-\lambda f$ admits solutions for $\lambda_n=\theta n$. The eigenfunctions are given by some linear modifications of Laguerre polynomials (see \cite{schoutens2012stochastic}) $L_n^{(b-1)}(ax)$ for $x>0$ and $n \in \N$, where the Laguerre polynomials $L_n^{(\gamma)}(x)$ are defined by the Rodrigues formula (see \cite{ismail2005classical})
\begin{equation*}
L_n^{(\gamma)}(x)=\frac{1}{n!}x^{-\gamma}e^x\dersup{}{x}{n}x^{n+\gamma}e^{-x}, \ x \in \R, \ \gamma>-1, \ n \in \N_0,
\end{equation*}
with normalization constants $K_n=\sqrt{\frac{\Gamma(b)n!}{\Gamma(b+n)}}$.
\subsubsection{The Jacobi process}
The Jacobi process is solution of the SDE
\begin{equation*}
dX(t)=-\theta\left(X(t)-\frac{b-a}{a+b+2}\right)dt+\sqrt{\frac{2\theta}{a+b+2}(1-X^2(t))}dW(t), \ t \ge 0
\end{equation*}
as $a,b>-1$. The diffusion space is $E=(-1,1)$ and its stationary density is a Beta one
\begin{equation*}
m(x)=(1-x)^a(1+x)^b\frac{\Gamma(a+b+2)}{\Gamma(a+1)\Gamma(b+1)2^{a+b+1}}.
\end{equation*}
The eigenvalue equation $\cG f=-\lambda f$ admits solutions for $\lambda_n=\frac{n \theta (n+a+b+1)}{a+b+2}$ and the eigenfunctions are Jacobi polynomials defined by the Rodrigues formula (see \cite{ismail2005classical})
\begin{equation*}
P^{(a,b)}_n(x)=\frac{(-1)^n}{2^nn!}(1-x)^{-a}(1+x)^{-b}\dersup{}{x}{n}[(1-x)^{a+n}(1+x)^{b+n}]
\end{equation*}
with normalization constants $K_n=\sqrt{\frac{2^{a+b+1}\Gamma(n+a+1)\Gamma(n+b+1)}{(2n+a+b+1)\Gamma(n+a+b+1)n!}}$. Let us recall that some properties of statistical interest for the Jacobi process have been highlighted in \cite{demni2009large}.
\subsection{Pearson diffusions of spectral category II}\label{subsecspecII}
\subsubsection{The FS process}
The FS process is solution of the SDE
\begin{equation*}
dX(t)=-\theta\left(X(t)-\frac{\beta}{\beta-2}\right)dt+\sqrt{\frac{4\theta}{\alpha(\beta-2)}X(t)(\alpha X(t)+\beta)}dW(t), \ t \ge 0
\end{equation*}
as $\alpha,\theta>0$ and $\beta>2$. The diffusion space is $E=(0,+\infty)$ and its stationary density is a Fisher-Snedecor one:
\begin{equation*}
m(x)=\frac{\left(\frac{\alpha x}{\alpha x +\beta}\right)^{\frac{\alpha}{2}}\left(\frac{\beta}{\alpha x +\beta}\right)^{\frac{\beta}{2}}}{xB\left(\frac{\alpha}{2},\frac{\beta}{2}\right)}, \ x>0.
\end{equation*}
A spectral analysis of the FS process has been carried in \cite{avram2013spectralfs}. In particular it is not difficult to see that the Fisher-Snedecor density admits finite even moments up to $2N_1$ as $N_1=\lfloor\frac{\beta}{4}\rfloor$. Thus we have a finite number of simple eigenvalues in the discrete spectrum of $-\cG$. The eigenvalues are given by
\begin{equation*}
\lambda_n=\frac{\theta}{\beta-2}n(\beta-2n), \ n=0,\dots, \left\lfloor\frac{\beta}{4}\right\rfloor
\end{equation*}
and the respective eigenfunctions are the Fisher-Snedecor polynomials $F_n^{(\alpha,\beta)}(x)$, defined by the Rodrigues formula
\begin{equation*}
F_n^{(\alpha,\beta)}(x)=x^{1-\frac{\alpha}{2}}(\alpha x+ \beta)^{\frac{\alpha}{2}+\frac{\beta}{2}}\dersup{}{x}{n}\left[2^nx^{\frac{\alpha}{2}+n-1}(\alpha x+\beta)^{n-\frac{\alpha}{2}-\frac{\beta}{2}}\right]
\end{equation*}
with normalizing constant
\begin{equation*}
K_n=(-1)^n\sqrt{\frac{B\left(\frac{\alpha}{2},\frac{\beta}{2}\right)}{n! (2\beta)^{2n}B\left(\frac{\alpha}{2}+n,\frac{\beta}{2}-2n\right)}\left[\prod_{k=1}^{n}\left(\frac{\beta}{2}+k-2n\right)^{-1}\right]}.
\end{equation*}
The absolutely continuous spectrum $\sigma_{ac}(-\cG)=(\Lambda,+\infty)$ where the cut-off $\Lambda$ is given by
\begin{equation*}
\Lambda_1=\frac{\theta \beta^2}{8(\beta-2)}.
\end{equation*}
For $\lambda \in (\Lambda_1,+\infty)$, the fundamental solutions of the Sturm-Liouville equation $\cG f=-\lambda f$ are given in terms of hypergeometric functions. A detailed study is made in \cite{avram2013spectralfs}. In any case, the solution that appears in the absolutely continuous part of the spectral decomposition is
\begin{equation*}
f_1(x,-\lambda)={}_2F_1\left(-\frac{\beta}{4}+\Delta_1(\lambda),-\frac{\beta}{4}-\Delta_1(\lambda);\frac{\alpha}{2};-\frac{\alpha}{\beta}x\right)
\end{equation*}
where
\begin{equation*}
\Delta_1(\lambda)=\sqrt{\frac{\beta^2}{16}-\frac{\lambda (\beta-s)}{2\theta}}.
\end{equation*}
Let us recall that the solution $f_1$ is found by making use of the theory of Kummer's solutions for Hypergeometric Equations (see \cite{slater1966generalized,prosser1994kummer}). 
\subsubsection{The RG process}
The RG process is solution of the SDE
\begin{equation*}
dX(t)=-\theta \left(X(t)-\frac{\alpha}{\beta-1}\right)dt+\sqrt{\frac{2\theta}{\beta-1}X^2(t)}dW(t), \ t \ge 0
\end{equation*}
as $\alpha,\theta>0$ and $\beta>1$. The diffusion space is $E=(0,+\infty)$ and its stationary density is a reciprocal Gamma one:
\begin{equation*}
m(x)=\frac{\alpha^\beta}{\Gamma(\beta)}x^{-\beta-1}e^{-\frac{\alpha}{x}}, \ x>0.
\end{equation*}
As for the Fisher-Snedecor distribtuon, let us recall that such density admits finite even moments up to $2N_2$ as $N_2=\left\lfloor\frac{\beta}{2}\right\rfloor$. A complete spectral analysis of the RG process has been made in \cite{leonenko2010statistical}. We have a finite number of simple eigenvalues in the discrete spectrum of $-\cG$, given by
\begin{equation*}
\lambda_n=n\theta\frac{\beta-n}{\beta-1}, \ n=0,\dots, \left\lfloor\frac{\beta}{2}\right\rfloor.
\end{equation*} 
The eigenfunctions are Bessel polynomials $B_n^{(\alpha,\beta)}$ defined by the Rodrigues formula
\begin{equation*}
B_n^{(\alpha,\beta)}(x)= x^{\beta+1}e^{\frac{\alpha}{x}}\dersup{}{x}{n}\left[x^{2n-(\beta+1)}e^{-\frac{\alpha}{2}}\right],
\end{equation*}
with normalizing constant
\begin{equation*}
K_n=\frac{(-1)^n}{\alpha^n}\sqrt{\frac{(\beta-2n)\Gamma(\beta)}{\Gamma(n+1)\Gamma(\beta-n+1)}}.
\end{equation*}
The absolutely continuous spectrum $\sigma_{ac}(-\cG)=(\Lambda,+\infty)$ where the cut-off $\Lambda$ is given by
\begin{equation*}
\Lambda_2=\frac{\theta \beta^2}{4(\beta-1)}.
\end{equation*}
For $\lambda \in (\Lambda_2,+\infty)$, the eigenfunction we will use is given by
\begin{equation*}
f_2(x,-\lambda)=\alpha^{\frac{\beta+1}{2}}{}_2F_0\left(-\frac{\beta}{2}+\Delta_2(\lambda),-\frac{\beta}{2}-\Delta_2(\lambda);\, ;-\frac{x}{\alpha}\right)
\end{equation*}
where
\begin{equation*}
\Delta_2(\lambda)=\frac{1}{2}\sqrt{\beta^2-\frac{4\lambda(\beta-1)}{\theta}}.
\end{equation*}
\subsubsection{Spectral decomposition theorem for spectral category II}
Let us recall the spectral decomposition theorem for Pearson diffusions of spectral category II as given in \cite{avram2013spectralfs,leonenko2010statistical}.
\begin{theorem}\label{spedecRG}
	Let $X(t)$ be a Pearson diffusion of spectral category II and, if $X(t)$ is a FS process, let $\alpha>2$ with $\alpha \not = 2(m+1)$ for any $m \in \N$. The density $p(t,x;x_0)$ admits the following spectral decomposition:
	\begin{equation*}
	p(t,x;x_0)=p_d(t,x;x_0)+p_c(t,x;x_0)
	\end{equation*}
	where
	\begin{equation*}
	p_d(t,x;x_0)=m(x)\sum_{n=0}^{N_j}e^{-\lambda_n t}Q_n(x_0)Q_n(x)
	\end{equation*}
	and
	\begin{equation*}
	p_c(t,x;x_0)=\frac{m(x)}{\pi}\int_{\Lambda_j}^{+\infty}e^{-\lambda t}a_j(\lambda)f_j(x_0,-\lambda)f_j(x,-\lambda)d\lambda,
	\end{equation*}
	where $j=1,2$ and
	\begin{align*}
	a_1(\lambda)&=(-i\Delta_1(\lambda))\left|\frac{\sqrt{B\left(\frac{\alpha}{2},\frac{\beta}{2}\right)}\Gamma\left(-\frac{\beta}{4}+\Delta_1(\lambda)\right)\Gamma\left(\frac{\alpha}{2}+\frac{\beta}{4}+\Delta_1(\lambda)\right)}{\Gamma\left(\frac{\alpha}{2}\right)\Gamma(1+2\Delta_1(\lambda))}\right|^2;\\
	a_2(\lambda)&=(-i\Delta_2(\lambda))\left|\frac{\sqrt{\Gamma(\beta)}\Gamma\left(-\frac{\beta}{2}+\Delta_2(\lambda)\right)}{\alpha^{\frac{\beta+1}{2}}\Gamma(1+2\Delta_2(\lambda))}\right|^2.
	\end{align*}
\end{theorem}
\subsection{Pearson diffusions of spectral category III: Student diffusions}\label{Sec23}
The Student process is solution of the SDE
\begin{equation*}
dX(t)=-\theta(X(t)-\mu)dt+\sqrt{\frac{2\theta \delta^2}{\nu-1}\left(1+\left(\frac{X(t)-\mu'}{\delta}\right)^2\right)}dW(t), \ t \ge 0
\end{equation*}
where $\theta,\delta>0$, $\mu,\mu' \in \R$ and $\nu>1$. If $\mu=\mu'$ we refer to it as the symmetric Student process, while for $\mu \not = \mu'$ it is called skew Student process. \\
In \cite{avram2013spectral} it is shown that such process admits diffusion space $E=\R$ and stationary density
\begin{equation*}
m(x) = \frac{\Gamma \left(\frac{\nu+1}{2}\right)}{\delta \sqrt{\pi} \Gamma \left(\frac{\nu}{2} \right)} \, \prod\limits_{k=0}^{\infty}\left( 1 + \left( \frac{(\mu - \mu^{\prime})(\nu-1)}{\delta(\nu+1+2k)}\right)^{2} \right)^{-1}\frac{\exp \left\{ \frac{(\mu -\mu^{\prime })(\nu-1)}{\delta}\arctan{\left( \frac{x - \mu^{\prime}}{\delta } \right)} \right\}}{\left[ 1 + \left( \frac{x - \mu^{\prime}}{\delta}\right)^{2} \right]^{\frac{\nu+1}{2}}}, \ x \in \R.
\end{equation*}
The Student density admits finite even moments up to $2N_3$ as $N_3=\left\lfloor \frac{\nu}{2}\right\rfloor$. Thus the discrete spectrum of $(-\cG)$ is made of a finite number of simple eigenvalues given by
\begin{equation*}
\lambda_n=\frac{n\theta(\nu-n)}{\nu-1}, \ n=0,\dots, \left\lfloor\frac{\nu}{2}\right\rfloor
\end{equation*}
while the eigenfunctions are the generalized Routh-Romanovski polynomials $\widetilde{R}_n(x)$ defined by the Rodrigues formula
\begin{align*}
\widetilde{R}_n(x)&=\left(\frac{\delta^2}{\nu-1}\right)^n\left(1+\left(\frac{x-\mu'}{\delta}\right)^2\right)^{\frac{\nu+1}{2}} \exp\left\{\frac{(\mu'-\mu)(\nu-1)}{\delta}\arctan\left(\frac{x-\mu'}{\delta}\right)\right\}\\
&\times \dersup{}{x}{n}\left(1+\left(\frac{x-\mu'}{\delta}\right)^2\right)^{n-\frac{\nu+1}{2}}\exp\left\{\frac{(\mu-\mu')(\nu-1)}{\delta}\arctan\left(\frac{x-\mu'}{\delta}\right)\right\}
\end{align*}
with normalization constant
\begin{multline*}
K_n=\left(\frac{1-\nu}{2\delta}\right)^n\sqrt{\frac{(2n-\nu-2)\sin[\pi(2n-\nu+1)]\Gamma(\nu+n-1)\Gamma^2\left(\frac{\nu+1}{2}-n\right)}{(\nu-1)(-1)^nn!\delta 2^{1-\nu}\pi^2 c(\nu,\delta,\mu,\mu')}}\\\times \sqrt{\prod_{k=0}^{+\infty}\left[1+\left(\frac{(\nu-1)(\mu-\mu')}{\delta(\nu+1-2n+2k)}\right)^{2}\right]^{-1}},
\end{multline*}
where $c(\nu,\delta,\mu,\mu')$ is a suitable constant. If $\mu=\mu'$ we obtain the classical Routh-Romanovski polynomials $R_n(x)$ and their normalizing constant (up to a multiplicative constant). Concerning the absolutely continuous spectrum $\sigma_{ac}(-\cG)=(\Lambda_3,+\infty)$, where the cut-off $\Lambda_3$ is given by
\begin{equation*}
\Lambda_3=\frac{\theta \nu^2}{4(\nu-1)},
\end{equation*}
we have to recall that its elements are of multiplicity $2$.\\
Concerning the eigenfunctions, in place of the monotonic solutions considered in \cite{avram2013spectral}, we can consider two independent linear combination of them. Indeed, by using the same strategy adopted in \cite{leonenko2010statisticalb}, referring to Kummer solutions (see \cite{prosser1994kummer,slater1966generalized}), we can consider one of the solutions of $\cG f=-\lambda f$ and its complex conjugate. In \cite{leonenko2010statisticalb} the eigenfunctions are explicitly expressed in the case $\mu=\mu'$, while for $\mu \not = \mu'$, the eigenfunctions are studied in \cite{arista2020explicit}, by reconsidering the seminal paper \cite{wong1964construction}. For the sake of shortness, we omit the explicit formula for such eigenfunctions, but we refer to it as $f_3(x;-\lambda)$ and $\overline{f}_3(x;-\lambda)$. Using Linetski approach (see \cite{linetsky2007spectral}) as done in \cite{avram2013spectral}, we obtain the following spectral decomposition theorem.
\begin{theorem}
	The Student process is ergodic for $\nu \not = 2k-1$ for any $k \in \N$. Under such choice of parameters, the density $p(t,x;x_0)$ admits the following spectral decomposition:
	\begin{equation*}
	p(t,x;x_0)=p_d(t,x;x_0)+p_c(t,x;x_0)
	\end{equation*}
	where
	\begin{equation*}
	p_d(t,x;x_0)=m(x)\sum_{n=0}^{N_3}e^{-\lambda_n t}Q_n(x_0)Q_n(x)
	\end{equation*}
	and
	\begin{multline*}
	p_c(t,x;x_0)=m(x)\int_{\Lambda_3}^{+\infty}e^{-\lambda t}\left(\frac{f_3(x_0,-\lambda)f_3(x,-\lambda)}{\Norm{f_3(\cdot,-\lambda)}{L^2(m(dx))}^2}+\frac{\bar{f}_3(x_0,-\lambda)\bar{f}_3(x,-\lambda)}{\Norm{\bar{f}_3(\cdot,-\lambda)}{L^2(m(dx))}^2}\right.\\\left.+\frac{f_3(x_0,-\lambda)\bar{f}_3(x,-\lambda)+\bar{f}_3(x_0,-\lambda)f_3(x,-\lambda)}{\Norm{\bar{f}_3(\cdot,-\lambda)}{L^2(m(dx))}\Norm{f_3(\cdot,-\lambda)}{L^2(m(dx))}}\right)d\lambda.
	\end{multline*}
\end{theorem}
Let us also underline the expression of the speed density that will be useful in what follows
\begin{equation*}
sp(x)=\frac{ \exp\left\{\frac{(\mu -\mu')(\nu-1)}{\delta} \; \arctan\left(\frac{x-\mu'}{\delta}\right)\right\}} {\left(1+\left(\frac{x-\mu'}{\delta}\right)^{2}\right)^{\frac{\nu+1}{2}}}, x \in \mathbb{R},
\end{equation*}
stressing that $\int_{-\infty}^{+\infty}sp(x)dx=M<+\infty$.
\section{Inverse subordinators and non-local convolution derivatives}\label{Sec3}
Now let us introduce our main object of study. Let us denote by $\BF$ the convex cone of Bernstein functions, that is to say $\Phi \in \BF$ if and only if $\Phi \in C^\infty(\R^+)$, $\Phi(\lambda) \ge 0$ and for any $n \in \N$
\begin{equation*}
(-1)^n\dersup{\Phi}{\lambda}{n}(\lambda)\le 0.
\end{equation*}
In particular it is known that for $\Phi \in \BF$ the following L\'evy-Khintchine representation (\cite{schilling2012bernstein}) is given
\begin{equation}\label{LKrepr}
\Phi(\lambda)=a_\Phi+b_\Phi\lambda+\int_0^{+\infty}(1-e^{-\lambda t})\nu_\Phi(dt)
\end{equation}
where $a_\Phi,b_\Phi \ge 0$ and $\nu_\Phi$ is a L\'evy measure on $\R^+$ such that
\begin{equation}\label{intcond}
\int_0^{+\infty}(1 \wedge t)\nu_\Phi(dt)<+\infty.
\end{equation}
The triple $(a_\Phi,b_\Phi,\nu_\Phi)$ is called the L\'evy triple of $\Phi$. Also the vice versa can be shown, i.e. for any L\'evy triple $(a_\Phi,b_\Phi,\nu_\Phi)$ such that $\nu_\Phi$ is a L\'evy measure satisfying the integral condition \cref{intcond} there exists a unique Bernstein function $\Phi$ such that Equation \cref{LKrepr} holds. In the following we will consider $a_\Phi=b_\Phi=0$ and $\nu_\Phi(0,+\infty)=+\infty$. Moreover, let us denote by $\bar{\nu}_\Phi(t)=\nu_\Phi(t,+\infty)$.\\
It is also known (see \cite{schilling2012bernstein}) that for each Bernstein function $\Phi \in \BF$ there exists a unique subordinator $\sigma_\Phi=\{\sigma_\Phi(y), y \ge 0\}$ (i. e. an increasing L\'evy process) such that
\begin{equation*}
\E[e^{-\lambda \sigma_\Phi(y)}]=e^{-y\Phi(\lambda)}.
\end{equation*}
For general notions on subordinators we refer to \cite[Chapter $3$]{bertoin1996levy} and \cite{bertoin1999subordinators}. Our hypothesis on $a_\Phi,b_\Phi$ and $\nu_\Phi$ imply that $\sigma_\Phi$ is a pure jump strictly increasing process.\\
For any Bernstein function $\Phi$ we can define the inverse subordinator $L_\Phi$ as, for any $t>0$
\begin{equation*}
L_\Phi(t):=\inf\{y > 0: \ \sigma_\Phi(y)>t\}.
\end{equation*}
$L_\Phi(t)$ is absolutely continuous for any $t>0$ and we can denote by $f_\Phi(s;t)$ its density. Let us recall (see \cite{meerschaert2008triangular}) that, denoting by $\overline{f}_\Phi(s;\lambda)$ the Laplace transform of $f_\Phi(s;t)$ with respect to $t$,
\begin{equation*}
\overline{f}_\Phi(s;\lambda)=\frac{\Phi(\lambda)}{\lambda}e^{-s\Phi(\lambda)}, \ \lambda>0.
\end{equation*}
Moreover, let us observe that, under our assumptions, the sample paths of $L_\Phi(t)$ are almost surely increasing and continuous.
Now let us recall the definition of non-local convolution derivative, defined in \cite{kochubei2011general} and \cite{toaldo2015convolution}.
\begin{definition}
	Let $u:\R^+ \to \R$ be an absolutely continuous function. Then we define the non-local convolution derivative induced by $\Phi$ of $u$ as
	\begin{equation}\label{Cap}
	\partial_t^\Phi u(t)=\int_0^t u'(\tau)\overline{\nu}_\Phi(t-\tau)d\tau.
	\end{equation}
\end{definition}
Let us observe that one can define also the regularized version of the non-local convolution derivative as
\begin{equation}\label{RegCap}
\partial_t^\Phi u(t)=\der{}{t}\int_0^t  (u(\tau)-u(0+))\overline{\nu}_\Phi(t-\tau)d\tau
\end{equation}
observing that it coincides with the previous definition on absolutely continuous functions.\\
It can be shown, by Laplace transform arguments (see, for instance \cite{kochubei2019growth,ascione2020generalized}) or by Green functions arguments (see \cite{kolokol2019mixed}), that the Cauchy problem
\begin{equation*}
\begin{cases}
\partial_t^\Phi\fe_\Phi(t;\lambda)=\lambda \fe_\Phi(t;\lambda) & t>0\\
\fe_\Phi(0;\lambda)=1
\end{cases}
\end{equation*}
admits a unique solution for any $\lambda \in \R$ and it is given by $\fe_\Phi(t;\lambda):=\E[e^{\lambda L_\Phi(t)}]$. In \cite{ascione2020non} the following proposition is proved.
\begin{proposition}
	Fix $t>0$. Then there exists a constant $K(t)$ such that
	\begin{equation}\label{unifest}
	\lambda \fe_\Phi(t;-\lambda)\le K(t), \ \forall \lambda \in [0,+\infty).
	\end{equation}
\end{proposition}
A Bernstein function $\Phi$ is said to be complete if its L\'evy measure $\nu_\Phi(dt)$ admits a density $\nu_\Phi(t)$ that is completely monotone. Following the approach given in \cite{baeumer2001stochastic} we are able to prove the following theorem (the proof is given in Appendix).
\begin{theorem}\label{thmsemig}
	Let $(X,\Norm{\cdot}{})$ be a Banach space and $(T(t))_{t \ge 0}$ be a uniformly bounded and strongly continuous $C_0$-semigroup on $X$. Define the family of linear operators on $X$ $(T_\Phi(t))_{t \ge 0}$ as
	\begin{equation*}
	T_\Phi(t)u=\int_0^{+\infty}T(s)uf_\Phi(s;t)ds, \ u \in X,
	\end{equation*}
	where $\Phi$ is a driftless complete Bernstein function, $f_\Phi(s;t)$ is the density of the inverse subordinator $L_\Phi(t)$ associated to $\Phi$ and the integral has to be intended in Bochner sense. Then $(T_{\Phi}(t))_{t \ge 0}$ is a uniformly bounded and strongly continuous family of linear operators. Moreover, it is also strongly analytic in a suitable sector $\C(\alpha)=\{z \in \C \setminus \{0\}: \ |{\rm Arg}(z)|<\alpha\}$ where ${\rm Arg}$ is the principal argument. Finally, if $(A,\cD(A))$ is the generator of the semigroup $(T(t))_{t \ge 0}$ and $u \in \cD(A)$, then also $T_\Phi(t)u \in \cD(A)$ and it solves the Cauchy problem
	\begin{equation}\label{abscau}
	\begin{cases}
	\partial_t^\Phi T_\Phi(t)u=AT_\Phi(t)u, & t>0,\\
	T_\Phi(0)u=u,
	\end{cases}
	\end{equation}
	where the equality holds in $X$ (and not necessarily pointwise).
\end{theorem}
\begin{rmk}
	Let us recall that $T_\Phi(t)$ in general is not a semigroup.
\end{rmk}
Concerning, strong continuity, we refer to the definition given in \cite{pazy2012semigroups} for general families of operators of one parameter. Moreover, let us observe that a similar result has been shown in \cite[Theorem $2.1$]{chen2017time}. However, following the lines of \cite{baeumer2001stochastic}, we are also able to show better regularity of the family of operators $(T_\Phi(t))_{t \ge 0}$.\\
Let us give some examples of complete Bernstein functions and associated subordinators.
\begin{itemize}
	\item For fixed $\alpha \in (0,1)$ we can consider $\Phi(\lambda)=\lambda^\alpha$. In this case $\sigma_\Phi$ is an $\alpha$-stable subordinator. Extensive informations on inverse $\alpha$-stable subordinators are given in \cite{meerschaert2013inverse}. In this case $\fe_\Phi(t;\lambda)=E_\alpha(\lambda t^\alpha)$, where $E_\alpha$ is the one-parameter Mittag-Leffler function (see \cite{bingham1971limit}) defined as
	\begin{equation*}
	E_\alpha(z)=\sum_{k=0}^{+\infty}\frac{z^k}{\Gamma(\alpha k+1)}, \ z \in \C.
	\end{equation*}
	In such case, the Caputo-type non-local derivative coincides with the classical Caputo fractional derivative (see \cite{kilbas2006theory}).
	\item For fixed $\alpha \in (0,1)$ and $\theta>0$ we can consider $\Phi(\lambda)=(\lambda+\theta)^\alpha-\theta^\alpha$. In this case $\sigma_\Phi$ is a tempered $\alpha$-stable subordinator. Inverse tempered stable subordinators are studied for instance in \cite{kumar2015inverse}. Even in this case the Caputo-type non-local derivative is linked to a well-known non-local operator, called the tempered fractional derivative (see \cite{cao2014tempered});
	\item For fixed $\alpha \in (0,1)$ we can consider $\Phi(\lambda)=\log(1+\lambda^{\alpha})$. In this case $\sigma_\Phi$ is a geometric $\alpha$-stable subordinator. For informations on such process we refer to \cite{vsikic2006potential}. 
	\item If in the previous example we consider $\alpha=1$ we obtain $\Phi(\lambda)=\log(1+\lambda)$ that is still a complete Bernstein function. In this case $\sigma_\Phi$ is a Gamma subordinator. Again we refer to \cite{vsikic2006potential} and references therein.
\end{itemize}

\section{Definition of the non-local Pearson diffusions}\label{Sec4}
Let us now define the non-local Pearson diffusions.
\begin{definition}
	Let $X(t)$ be a Pearson diffusion and $\Phi \in \BF$ be a driftless Bernstein function. Let $L_\Phi(t)$ be an inverse subordinator associated to $\Phi$ and independent of $X(t)$. Then we define the \textbf{non-local Pearson diffusion} induced by $X$ and $L_\Phi$ as $X_\Phi(t):=X(L_\Phi(t))$.
\end{definition}
The first property one has to recall is that $X_\Phi(t)$ is not a Markov process, but it is still a semi-Markov one (see \cite{cinlar}).\\
The fact that $L_\Phi(t)$ is a.s. continuous implies that the Brownian motion $W(t)$ in the stochastic differential equation \cref{PSDE} is in synchronization (in the sense of \cite{kobayashi2011stochastic}) with $L_\Phi(t)$ for any driftless Bernstein function $\Phi$. Thus the first change-of-variable formula for synchronized processes (see \cite[Lemma $2.3$]{kobayashi2011stochastic}) and the duality theorem \cite[Theorem $4.2$]{kobayashi2011stochastic} lead us to the following characterization of non-local Pearson diffusions (following the lines of \cite{leonenko2017heavy}). 
\begin{proposition}
	$X_\Phi(t)$ is the unique strong solution of
	\begin{equation}\label{tcPSDE}
	dX_\Phi(t)=\mu(X_\Phi(t))dL_\Phi(t)+\sigma(X_\Phi(t))dW(L_\Phi(t))
	\end{equation}
	with initial datum $X_\Phi(0)=x_0$ if and only if $X_\Phi(t)$ is a non-local Pearson diffusion.
\end{proposition}
Thus we can characterize the non-local Peason diffusions as the unique strong solutions of Equation \cref{tcPSDE}, in analogy of what is done with the classical Pearson diffusion. In particular, from this result, we still have that non-local Pearson diffusions are semimartingales, but with respect to a time-changed filtration $\cF_{L_\Phi(t)}$.\\ 
Now let us show that non-local Pearson diffusions admit in some sense a transition probability density.
\begin{lemma}\label{lem:intrep}
	Let $X_\Phi(t)$ be a non-local Pearson diffusion with diffusion space $E$. Then there exists a function $p_\Phi(t,x;x_0)$ such that for any Borel set $B \in \cB(E)$, $t>0$ and $x_0 \in E$ it holds
	\begin{equation*}
	\bP(X_\Phi(t)\in B|X_\Phi(0)=x_0)=\int_{B}p_\Phi(t,x;x_0)dx.
	\end{equation*}
	Moreover, the following integral representation holds:
	\begin{equation}\label{intrep}
	p_\Phi(t,x;x_0)=\int_0^{+\infty}p(s,x;x_0)f_\Phi(s;t)ds, \ x,x_0 \in E, \ t>0
	\end{equation}
	where $p$ is the transition probability density of $X(t)$.
\end{lemma}
\begin{proof}
	Let us first observe that since $L_\Phi(0)=0$ almost surely, then $X_\Phi(0)=X(0)$ almost surely. Now fix $B \in \cB(E)$ and observe, by the independence of $L_\Phi$ and $X$ and by definition of $p(t,x;x_0)$,
	\begin{align*}
	\bP(X_\Phi(t) \in B|X_\Phi(0)=x_0)&=\int_0^{+\infty}\bP(X(s) \in B|X_\Phi(0)=x_0)f_\Phi(s;t)ds\\
	&=\int_0^{+\infty}\int_{B}p(s,x;x_0)dxf_\Phi(s;t)ds\\
	&=\int_{B}\int_0^{+\infty}p(s,x;x_0)f_\Phi(s;t)dsdx,
	\end{align*}
	where we used Fubini's theorem since all the integrands are non-negative. Thus, by uniqueness of the Radon-Nikodym derivative of measures, we have Equation \cref{intrep}. 
\end{proof}
We call $p_\Phi(t,x;x_0)$ the transition probability density of $X_\Phi(t)$.\\
Let us stress out that, for any Pearson diffusion, if we consider the semigroup \linebreak $T(t)g(x)=\E^x[g(X(t))]$ acting on $C_0(E)$, as a direct consequence of the previous lemma, the family of operators $(T_\Phi(t))_{t \ge 0}$ acting on $C_0(E)$ and defined as $T_\Phi(t)g(x)=\E^x[g(X_\Phi(t))]$ is obtained by time-changing $T(t)$. Moreover, we can characterize the family of adjoint operator $(T^*_\Phi(t))_{t \ge 0}$ of $T_\Phi(t)$.
\begin{lemma}\label{lem:adj}
	Let $X_\Phi(t)$ be a non-local Pearson diffusion and $T_\Phi(t)$ be the family of operators defined on $C_0(E)$ as $T_\Phi(t)g(x)=\E^x[g(X_\Phi(t))]$. Then the adjoint operators $(T^*_\Phi(t))_{t \ge 0}$ are defined as
	\begin{equation*}
	T^*_\Phi(t)f(x)=\int_{E}p_\Phi(t,x;y)f(y)dy
	\end{equation*}
	and represent the time-changed semigroup of the strongly continuous semigroup \linebreak $(T^*(t))_{t \ge 0}$ defined as
	\begin{equation*}
	T^*(t)f(x)=\int_{E}p(t,x;y)f(y)dy.
	\end{equation*}
\end{lemma}
We omit the proof since it easily follows from the definition of $T^*_\Phi$ and the integral representation given in \cref{lem:intrep}.\\
In the following sections we will investigate the spectral decomposition of the transition probability density of non-local Pearson diffusions and the existence of strong solutions for the associated backward and forward Kolmogorov problems.
\section{Spectral decomposition of non-local Pearson diffusions of spectral category I}\label{Sec5}
\subsection{Spectral decomposition of the transition probability density}
Let us consider a Pearson diffusion $X(t)$ of spectral category I with diffusion space $E$, generator $\cG$, stationary density $m(x)$ and orthonormal polynomials $Q_n(x)$. Then it is well known that the spectral decomposition of $p(t,x;x_0)$ is given by
\begin{equation*}
p(t,x;x_0)=m(x)\sum_{n=0}^{+\infty}e^{-\lambda_n t}Q_n(x)Q_n(x_0)
\end{equation*}
where $\lambda_n$ are the eigenvalues of $-\cG$ for each $Q_n$. We want to show a similar spectral decomposition for $p_\Phi(t,x;x_0)$ where, in place of $e^{-\lambda_n t}$, we have $\fe_\Phi(t;-\lambda_n)$. To do this, we first need to show the following technical Lemma.
\begin{lemma}\label{lem:contrseries}
	Let $\Phi \in \BF$ be a driftless Bernstein function, $X(t)$ be a Pearson diffusion of spectral category I with diffusion space $E$, associated family of classical orthonormal polynomials $Q_n(x)$ and $\lambda_n$ the respective eigenvalues. Then the series
	\begin{equation*}
	\sum_{n=0}^{+\infty}\fe_\Phi(t;-\lambda_n)Q_n(x)Q_n(y)
	\end{equation*}
	absolutely converges for fixed $t>0$ and $x,y \in E$ and totally converges for $x,y$ belonging to compact sets in $E$ and $t \ge t_0$.
\end{lemma}
\begin{proof}
	Let us recall that, by Equation \cref{unifest}, for fixed $t>0$ there exists $K>0$ such that $\fe_\Phi(t;-\lambda_n)\le \frac{K}{\lambda_n}$.\\
	Let us first work with the OU process. We can suppose without loss of generality that $\mu=0$ and $\sigma=1$. Then there exists a constant $C>0$ such that
	\begin{equation}\label{normHermest}
	|Q_n(x)|\le C e^{\frac{x^2}{4}}n^{-\frac{1}{4}}\left(1+\left(\frac{|x|}{\sqrt{2}}\right)^{\frac{5}{2}}\right),
	\end{equation}
	where the estimate follows from the definition of the normalizing constant and the estimate on Hermite polynomials given in \cite[Page $369$]{sansone1959orthogonal}. Moreover, $\lambda_n=\theta n$, thus
	\begin{equation}\label{bound1}
	\fe_\Phi(t;-\lambda_n)|Q_n(x)||Q_n(y)|\le C n^{-1-\frac{1}{2}},
	\end{equation}	
	obtaining the absolute convergence. Total convergence follows from the same estimates and the fact that $\fe_\Phi(t;-\lambda_n)\le \fe_\Phi(t_0;-\lambda_n)$.\\
	Concerning the CIR process, we can suppose without loss of generality $a=1$. We have, considering the normalizing constant and the estimate on Laguerre polynomials given in \cite[Page $348$]{sansone1959orthogonal}, that there exists a constant $C$ independent of $x$ and $n$ such that, for $n$ big enough, it holds
	\begin{equation}\label{normLagest}
	|Q_n(x)|\le C \frac{e^{\frac{x}{2}}}{x^{\frac{2b-1}{4}}}n^{-\frac{1}{4}},
	\end{equation}
	thus we have again \cref{bound1} and both absolute and total converge follow.\\
	Now let us consider the Jacobi process case. By \cite[Theorem $8.21.8$]{szeg1939orthogonal} and the definition of the normalizing constant we have
	\begin{equation}\label{normJacest}
	Q_n(x)=C(x,a,b)\cos(N(a,b)\theta+\gamma(a))+O(n^{-1})
	\end{equation}
	where the remainder term and the function $C(x,a,b)$ are uniform in $x$ as $x \in K_x \subset (-1,1)$ with $K_x$ compact set and $\cos(\theta)=x$. Thus, in this case, we have
	\begin{equation*}
	\fe_\Phi(t;-\lambda_n)|Q_n(x)||Q_n(y)|\le C n^{-2},
	\end{equation*}
	the bound being uniform as $t \ge t_0>0$, $x \in K_x$ and $y \in K_y$ where $K_x$ and $K_y$ are compact sets in $E=(-1,1)$. 
\end{proof}
Now we are ready to give the spectral decomposition of the transition probability density $p_\Phi(t,x;x_0)$.
\begin{theorem}\label{thm:specI}
	Let $\Phi \in \BF$ be a driftless Bernstein function, $X(t)$ be a Pearson diffusion of spectral category I with diffusion space $E$, associated family of classical orthonormal polynomials $Q_n(x)$ with eigenvalues $\lambda_n$. Let $X_\Phi(t)$ be the respective non-local Pearson diffusion. Then
	\begin{equation*}
	p_\Phi(t,x;x_0)=m(x)\sum_{n=0}^{+\infty}\fe_\Phi(t;-\lambda_n)Q_n(x)Q_n(x_0)
	\end{equation*}
	for any $t>0$ and $x,x_0 \in E$.
\end{theorem}
\begin{proof}
	By using Equation \cref{intrep} we know that
	\begin{equation*}
	p_\Phi(t,x;x_0)=m(x)\int_0^{+\infty}\sum_{n=0}^{+\infty}e^{-\lambda_n s}Q_n(x)Q_n(x_0)f_\Phi(s;t)ds.
	\end{equation*}
	Now we have to show that we can exchange the order of the integral with the summation. To do this, let us observe that
	\begin{equation*}
	\int_0^{+\infty}e^{-\lambda_n s}|Q_n(x)Q_n(x_0)|f_\Phi(s;t)ds = \fe_\Phi(t;-\lambda_n)|Q_n(x)Q_n(x_0)|<+\infty
	\end{equation*}
	by the previous Lemma and then we can use Fubini's theorem to conclude the proof.
\end{proof}
Now that we have the spectral decomposition of the transition density, we can focus on showing the existence of strong solutions to the non-local backward Kolmogorov equation under suitable assumptions on the initial datum.
\subsection{The non-local backward Kolmogorov equation}
Here we want to focus on the following non-local Cauchy problem
\begin{equation}\label{nonlocback}
\begin{cases}
\partial_t^\Phi u(t,y)=\cG u(t,y) & t>0, \ y \in E \\
u(0,y)=g(y) & y \in E
\end{cases}
\end{equation}
for suitable initial data $g$.
\begin{definition}\label{strongsol}
	A function $u(t,y)$ is a \textbf{strong solution} (in $L^2(m(x)dx)$) of the problem \cref{nonlocback} if and only if:
	\begin{itemize}
		\item $t\ge 0 \mapsto u(t,\cdot) \in L^2(m(x)dx)$ is strongly continuous;
		\item The function $\partial_t^\Phi u(t,y)$ is well-defined for any $t>0$ and $y \in E$;
		\item The convolution integral involved in $\partial_t^\Phi u(t,y)$ is a Bochner integral and the derivative is a strong derivative in $L^2(m(x)dx)$;
		\item The function $t>0 \mapsto \partial_t^\Phi u(t,\cdot) \in L^2(m(x)dx)$ is strongly continuous;
		\item For fixed $t>0$, $u(t,\cdot)\in C^2(E)$;
		\item The equations of \cref{nonlocback} hold pointwise for any $t>0$ and almost all $y \in E$.
	\end{itemize}
\end{definition}
Let us first show the next result, following the lines of \cite{leonenko2013fractional}.
\begin{theorem}\label{thm:strongsolback}
	Let $X(t)$ be a Pearson diffusion of spectral category I with diffusion space $E$, associated family of orthonormal polynomials $Q_n$ with respective eigenvalues $\lambda_n$, generator $\cG$ and stationary density $m(x)$. Let $\Phi \in \BF$ be a driftless Bernstein function. Let $g \in L^2(m(x)dx)$ be decomposed as $g(y)=\sum_{n=0}^{+\infty}g_nQ_n(y)$, where the series converges in $L^2(m(x)dx)$, absolutely for fixed $y \in E$ and uniformly on compact intervals $[y_1,y_2]\subset E$. Then the unique strong solution in $L^2(m(x)dx)$ of the problem \cref{nonlocback} is given by
	\begin{equation}\label{backsol1}
	u(t,y)=\sum_{n=0}^{+\infty}\fe_\Phi(t;-\lambda_n)Q_n(y)g_n.
	\end{equation}
	Moreover, $p_\Phi(t,x;y)$ is the fundamental solution of the problem \cref{nonlocback}, in the sense that, for any initial datum $g \in L^2(m(x)dx)$ such that $g=\sum_{n=0}^{+\infty}g_nQ_n$ converges in $L^2(m(x)dx)$, absolutely for fixed $y \in E$ and uniformly on compact interval $[y_1,y_2]\subset E$, it holds
	\begin{equation*}
	u(t,y)=\int_{E}p_\Phi(t,x;y)g(x)dx.
	\end{equation*}
\end{theorem}
\begin{proof}
	First of all, let us observe that the function in Equation \cref{backsol1} is well defined, by showing that the involved series converges in $L^2(m(x)dx)$. To do this, let us first define for $N \in \N$
	\begin{equation*}
	u_N(t,y)=\sum_{n=0}^{N}\fe_\Phi(t;-\lambda_n)Q_n(y)g_n.
	\end{equation*}
	By Parseval's identity, setting $N>M$, we have that
	\begin{equation*}
	\Norm{u_N(t,\cdot)-u_M(t,\cdot)}{L^2(m(x)dx)}^2=\sum_{n=M+1}^{N}\fe_\Phi^2(t;-\lambda_n)g^2_n\le\sum_{n=M+1}^{N}g^2_n.
	\end{equation*}
	However, still using Parseval's identity, we have $\sum_{n=0}^{+\infty}g_n^2=\Norm{g}{L^2(m(x)dx)}^2<+\infty$,
	thus concluding the convergence in $L^2(m(x)dx)$ of the sequence $u_N(t,\cdot)$ by Cauchy's criterion.\\
	Now let us show that $t \mapsto u(t,\cdot)$ is strongly continuous in $L^2(m(x)dx)$. Concerning strong continuity at $0$, let us fix $\varepsilon>0$ and consider $N>0$ such that $\sum_{n=N+1}^{+\infty}g_n^2<\varepsilon$. Thus we have, by Parseval's identity,
	\begin{equation*}
	\Norm{u(t,\cdot)-g(\cdot)}{L^2(m(x)dx)}^2< \sum_{n=0}^{N}g_n^2|\fe_\Phi^2(t;-\lambda_n)-1|+\varepsilon.
	\end{equation*}
	Taking the limit superior as $t \to 0^+$ we have
	\begin{equation*}
	\limsup_{t \to 0^+}\Norm{u(t,\cdot)-g(\cdot)}{L^2(m(x)dx)}^2 \le \varepsilon.
	\end{equation*}
	Since $\varepsilon>0$ is arbitrary, we have continuity at $0^+$. Strong continuity at any $t>0$ is proven in an analogous way.\\
	Now let us observe that for the single summand it holds
	\begin{equation}\label{singsumm}
	\partial_t^\Phi \fe_\Phi(t;-\lambda_n)Q_n(y)g_n=-\lambda_n\fe_\Phi(t;-\lambda_n)Q_n(y)g_n=\cG \fe_\Phi(t;-\lambda_n)Q_n(y)g_n
	\end{equation}
	for any $n \in \N$. Thus, we actually have to show that we can exchange the operators $\cG$ and $\partial_t^\Phi$ with the summation.\\
	Let us first consider $\partial_t^\Phi \sum_{n=0}^{+\infty}\fe_\Phi(t;-\lambda_n)Q_n(y)g_n$. First of all, let us denote by $\cI_\Phi(t)=\int_0^t \bar{\nu}_\Phi(s)ds$ the integrated tail of the L\'evy measure. $\cI_\Phi(t)$ is an increasing non-negative function with derivative $\bar{\nu}_\Phi(t)$, hence we can rewrite
	\begin{equation*}
	\int_0^t \bar{\nu}_\Phi(t-\tau)(u(\tau,y)-g(y))d\tau=\int_0^t (u(\tau,y)-g(y))d\cI_\Phi(t-\tau).
	\end{equation*}
	Now let us observe that
	\begin{equation}\label{totconv}
	u(\tau,y)-g(y)=\sum_{n=1}^{+\infty}(\fe_\Phi(\tau;-\lambda_n)-1)g_nQ_n(y)
	\end{equation}
	and in particular
	\begin{equation*}
	\sum_{n=1}^{+\infty}|(\fe_\Phi(\tau;-\lambda_n)-1)g_nQ_n(y)|\le \sum_{n=1}^{+\infty}|g_nQ_n(y)|,
	\end{equation*}
	where the right-hand side converges for fixed $y \in E$. Thus the series in the right-hand side of \cref{totconv} totally converges and we have, by using \cite[Theorem $7.16$]{rudin1964principles},
	\begin{equation}\label{intstep}
	\int_0^t \bar{\nu}_\Phi(t-\tau)(u(\tau,y)-g(y))d\tau=\sum_{n=0}^{+\infty}\int_0^t(\fe_\Phi(\tau;-\lambda_n)-1)g_nQ_n(y)\bar{\nu}_{\Phi}(t-\tau)d\tau. 
	\end{equation}
	Now we only need to show that we can exchange the derivative $\der{}{t}$ with the summation sign on the right-hand side of Equation \cref{intstep}. To do this, we have to show that the series of the derivatives converges uniformly in any compact set containing $t$ (see \cite[Theorem $7.17$]{rudin1964principles}). However, the series of the derivatives is given by
	\begin{equation*}
	\sum_{n=0}^{+\infty}-\lambda_n\fe_\Phi(t;-\lambda_n)g_nQ_n(y).
	\end{equation*}
	Now fix $t_0>0$ and $t\ge t_0$. Then we have, by Equation \cref{unifest} and the fact that $t \mapsto \fe_\Phi(t;-\lambda)$ is decreasing for any $\lambda>0$, $\lambda_n\fe_\Phi(t;-\lambda_n)\le M(t_0)$. In particular, this implies
	\begin{equation*}
	\sum_{n=0}^{+\infty}\lambda_n\fe_\Phi(t;-\lambda_n)|g_nQ_n(y)|\le M(t_0)\sum_{n=0}^{+\infty}|g_nQ_n(y)|,
	\end{equation*}
	and then the series of the derivative is totally convergent in any compact set separated from $0$. Finally, we have
	\begin{equation}\label{partialseries}
	\partial_t^\Phi u(t,y)=\sum_{n=0}^{+\infty}-\lambda_n\fe_\Phi(t;-\lambda_n)g_nQ_n(y).
	\end{equation}
	Arguing as we did for the series on the right-hand side of Equation \cref{backsol1}, one can easily show that the series on the right-hand side of Equation \cref{partialseries} strongly converges in $L^2(m(x)dx)$ and that $\partial_t^\Phi u(t,y)$ is strongly continuous in $L^2(m(x)dx)$ as $t>0$. \\
	Now we have to show that we can exchange the operator $\cG$ with the series on the right-hand side of Equation \cref{backsol1}. Since $\cG$ is a second order differential operator with polynomial coefficients, we only need to exchange the first and the second derivative with respect to $y$ with the series. To do this, we need to argue differently depending on the process. Let us first consider the OU process. The differential recurrence relation between Hermite polynomials (see, for instance, \cite[Formula $22.8.8$]{abramowitz1988handbook}) becomes, after normalization $Q_n'(y)=\sqrt{n}Q_{n-1}(y)$. Thus, combining this relation with \cref{normHermest} we have that, in any compact set $[y_1,y_2] \subset \R$ containing $y$, it holds
	\begin{equation}\label{eq:estder}
	|\fe_\Phi(t;-\lambda_n)g_nQ'_n(y)|\le C(t) n^{-\frac{3}{4}}|g_n|
	\end{equation}
	where the right-hand side is the summand of a convergent series, since both $(n^{-3/4})_{n \in \N}$ and $(g_n)_{n \in \N}$ belong to $\ell^2$. Concerning the second derivative, the Sturm-Liouville equation $\cG Q_n(y)=-\lambda_n Q_n(y)$ can be rewritten as $Q_n''(y)=yQ_n'(y)-nQ_n(y)$, thus we have to study only the term in $nQ_n(y)$. For such term, recalling that $\lambda_n=\theta n$, we know that $n\fe_\Phi(t;-\lambda_n)\le M(t)$ for some $M(t)$, hence the uniform convergence of $\sum_{n=0}^{+\infty}Q_n(y)g_n$ on compact intervals implies the uniform convergence of $\sum_{n=0}^{+\infty}\fe_\Phi(t;-\lambda_n)nQ_n(y)g_n$ on compact intervals.\\
	Now let us consider the CIR process. After normalization, \cite[Formula $5.1.14$]{szeg1939orthogonal} becomes \linebreak $Q'_{n,b-1}(y)=-\frac{(n-1)^{b/2}}{n^{(b-1)/2}}Q_{n-1,b}(y)$, where we denote with the indexes the dependence on the parameter $b$. By using \cref{normLagest}, we obtain again equation \cref{eq:estder} uniformly on compact intervals, thus obtaining the desired converges. Moreover, the Sturm-Liouville equation $\cG Q_n(y)=-\lambda_n Q_n(y)$ becomes $yQ_n''(y)=(y-b)Q_n'(y)-nQ_n(y)$ thus, since $\lambda_n=\theta n$, we can argue as in the OU process case to obtain theuniform convergence of $\sum_{n=0}^{+\infty}\fe_\Phi(t;-\lambda_n)nQ_n(y)g_n$ on compact intervals.\\
	Last case is the Jacobi process case. In this case, normalizing \cite[Formula $22.8.1$]{abramowitz1988handbook}, we have
	\begin{equation*}
	Q'_n(y)=\frac{n(a-b-(2n+a+b)y)}{(2n+a+b)(1-y^2)}Q_n(y)+\frac{2(n+a)(n+b)}{(2n+a+b)(1-y^2)}\sqrt{\frac{n}{n-1}}Q_{n-1}(y).
	\end{equation*}
	Both terms lead to a series of the form $\sum_{n=0}^{+\infty}n\fe_\Phi(t;-\lambda_n)Q_n(y)g_n$. However, this time we know that $Q_n(y)$ are bounded and $\lambda_n \sim Cn^2$, thus we have that (by still using \cref{unifest})
	\begin{equation*}
	|\fe_\Phi(t;-\lambda_n)nQ_n(y)g_n|\le C \frac{g_n}{n}
	\end{equation*}
	where the right-hand side is the summand of a convergent series since \linebreak $(g_n)_{n \in \N}, (1/n)_{n \in \N} \in \ell^2$. Concerning the second derivative, the Sturm-Liouville equation $\cG Q_n(y)=-\lambda_n Q_n(y)$ becomes $(1-y^2)Q_n''(y)=-((b-a)-(a+b-2)y)Q_n'(y)-n(n+a+b+1)Q_n(y)$,
	thus, since $n(n+a+b+1)$ and $\lambda_n$ are both quadratic in $n$, the same argument as in the previous cases leads to uniform convergence on compact intervals of $\sum_{n=0}^{+\infty}\fe_\Phi(t;-\lambda_n)n(n+a+b+1)Q_n(y)g_n$.\\
	Concluding we have in general
	\begin{equation*}
	\cG u(t,y)=\sum_{n=0}^{+\infty}\cG \fe_\Phi(t;-\lambda_n) Q_n(y)g_n
	\end{equation*}
	and then also $u(t,\cdot) \in C^2(E)$. Equation \cref{singsumm} and the fact we can exchange the order of the operators with the series imply that the equations of \cref{nonlocback} hold pointwise.\\
	Now let us show the third property in Definition \ref{strongsol}. First of all, let us show that the convolution integral is actually a Bochner integral. By Bochner's theorem (\cite[Theorem $1.1.4$]{arendtvector}) we only have to show that
	\begin{equation*}
		\int_0^t \Norm{u(t-\tau,\cdot)-u(0+,\cdot)}{L^2(m(x)dx)}\overline{\nu}_\Phi(\tau)d\tau<+\infty.
	\end{equation*}
	To do this, let us use Jensen's inequality and Parseval's identity to achieve
	\begin{align*}
		\left(\int_0^t \right.&\left.\vphantom{\int_0^t}\Norm{u(t-\tau,\cdot)-u(0+,\cdot)}{L^2(m(x)dx)}\overline{\nu}_\Phi(\tau)d\tau\right)^2\\&\le \cI_\Phi(t)\int_0^t \Norm{u(t-\tau,\cdot)-u(0+,\cdot)}{L^2(m(x)dx)}^2\overline{\nu}_\Phi(\tau)d\tau\\
		&=\cI_\Phi(t)\int_0^t \sum_{n=1}^{+\infty}(1-\fe_\Phi(t;-\lambda_n))^2g^2_n\overline{\nu}_\Phi(\tau)d\tau\le \cI^2_\Phi(t)\Norm{g}{L^2(m(x)dx)}^2.
	\end{align*}
	Now let us show that the derivative is actually a strong derivative in $L^2(m(x)dx)$, that is to say
	\begin{align}\label{limpass0}
	\begin{split}
	\lim_{h \to 0}\Norm{\frac{1}{h}\right.&\left. \vphantom{\frac{1}{h}}\left(\int_0^{t+h}(u(t+h-\tau,\cdot)-u(0+,\cdot))\overline{\nu}_\Phi(\tau)d\tau\right.\right.\\&\left.\left.\vphantom{\int_0^t}-\int_0^{t}(u(t-\tau,\cdot)-u(0+,\cdot))\overline{\nu}_\Phi(\tau)d\tau\right)-\partial_t^\Phi u(t,\cdot)}{L^2(m(x)dx)}=0.
	\end{split}	
	\end{align}
	Let us argue for $h>0$, as the case $h<0$ is analogous. By triangular inequality and Equation \eqref{partialseries} we have
	\begin{align}\label{limpass01}
		\begin{split}
		\Norm{\vphantom{\frac{1}{h}}\right.&\left. \frac{1}{h}\left(\int_0^{t+h}(u(t+h-\tau,\cdot)-u(0+,\cdot))\overline{\nu}_\Phi(\tau)d\tau-\int_0^{t}(u(t-\tau,\cdot)-u(0+,\cdot))\overline{\nu}_\Phi(\tau)d\tau\right)-\partial_t^\Phi u(t,\cdot)}{L^2(m(x)dx)}\\
		&\le \Norm{\frac{1}{h}\int_0^{t}(u(t+h-\tau,\cdot)-u(t-\tau,\cdot))\overline{\nu}_\Phi(\tau)d\tau+\sum_{n=1}^{+\infty}\lambda_n \fe_\Phi(t;-\lambda_n)g_nQ_n}{L^2(m(x)dx)}\\&+\Norm{\vphantom{\int_0^t}\frac{1}{h}\int_{t}^{t+h}(u(t+h-\tau,\cdot)-u(0+,\cdot))\overline{\nu}_\Phi(\tau)d\tau}{L^2(m(x)dx)}.
	\end{split}
	\end{align}
	 Let us argue with the second summand; by Bochner's theorem we have
	 \begin{align*}
	 	\Norm{\frac{1}{h}\int_{t}^{t+h}(u(t+h-\tau,\cdot)-u(0+,\cdot))\overline{\nu}_\Phi(\tau)d\tau}{L^2(m(x)dx)}&\le \frac{1}{h}\int_{t}^{t+h}\Norm{u(t+h-\tau,\cdot)-u(0+,\cdot)}{L^2(m(x)dx)}\overline{\nu}_\Phi(\tau)d\tau\\
	 	&\le \frac{\overline{\nu}_\Phi(t)}{h}\int_{0}^{h}\Norm{u(\tau,\cdot)-u(0+,\cdot)}{L^2(m(x)dx)}d\tau,
	 \end{align*}
 thus, taking the limit and recalling that, by strong continuity of $u(t,\cdot)$ in $L^2(m(x)dx)$ the function $t \mapsto \Norm{u(\tau,\cdot)-u(0+,\cdot)}{L^2(m(x)dx)}$ is continuous, we have
 \begin{equation}\label{limpass1}
 	\lim_{h \to 0^+}\Norm{\vphantom{\int_0^t}\frac{1}{h}\int_{t}^{t+h}(u(t+h-\tau,\cdot)-u(0+,\cdot))\overline{\nu}_\Phi(\tau)d\tau}{L^2(m(x)dx)}=0.
 \end{equation}
	Now let us consider the first summand of the right-hand side of Equation \eqref{limpass01}. To handle this summand, let us first observe that
	\begin{align}\label{limpass21}
		\begin{split}
		\int_0^t&(u(t+h-\tau,y)-u(t-\tau,y))\overline{\nu}_\Phi(\tau)d\tau\\&=\int_0^t\sum_{n=1}^{+\infty}(\fe_\Phi(t+h-\tau;-\lambda_n)-\fe_\Phi(t-\tau;-\lambda_n))g_nQ_n(y)\overline{\nu}_\Phi(\tau)d\tau.
	\end{split}
	\end{align}
	Being $|\fe_\Phi(t+h-\tau;-\lambda_n)-\fe_\Phi(t-\tau;-\lambda_n)|\le 1$ and $\sum_{n=1}^{+\infty}|g_nQ_n(y)|$ convergent, we have that the series on the right-hand side of Equation \eqref{limpass21} normally converges, thus we achieve
	\begin{align}\label{limpass22}
		\begin{split}
		\int_0^t&(u(t+h-\tau,y)-u(t-\tau,y))\overline{\nu}_\Phi(\tau)d\tau\\&=\sum_{n=1}^{+\infty}\int_0^t(\fe_\Phi(t+h-\tau;-\lambda_n)-\fe_\Phi(t-\tau;-\lambda_n))g_nQ_n(y)\overline{\nu}_\Phi(\tau)d\tau.
	\end{split}
	\end{align}
	Moreover, the same argument as before shows that the series on the right-hand side of Equation \eqref{limpass22} is absolutely convergent. Being also the series in Equation \eqref{partialseries} absolutely convergent, we have
	\begin{align}\label{limpass23}
		\begin{split}
		\Norm{\frac{1}{h}\int_0^{t}(u(t+h-\tau,\cdot)\vphantom{\frac{1}{h}}\right.&\left.-u(t-\tau,\cdot))\overline{\nu}_\Phi(\tau)d\tau+\sum_{n=1}^{+\infty}\lambda_n \fe_\Phi(t;-\lambda_n)g_nQ_n}{L^2(m(x)dx)}\\
		&=\Norm{\sum_{n=1}^{+\infty}\left(\int_0^{t}\frac{\fe_\Phi(t+h-\tau,\cdot)-\fe_\Phi(t-\tau,\cdot)}{h}\overline{\nu}_\Phi(\tau)d\tau+\lambda_n \fe_\Phi(t;-\lambda_n)\right)g_nQ_n}{L^2(m(x)dx)}.
	\end{split}
	\end{align}
	Now let us show that the series on the right-hand side of Equation \eqref{limpass23} converges in $L^2(m(x)dx)$. To do this, let us define
	\begin{equation*}
		S_N=\sum_{n=1}^{N}\left(\int_0^{t}\frac{\fe_\Phi(t+h-\tau,\cdot)-\fe_\Phi(t-\tau,\cdot)}{h}\overline{\nu}_\Phi(\tau)d\tau+\lambda_n \fe_\Phi(t;-\lambda_n)\right)g_nQ_n
	\end{equation*}
	and observe that, for any $N>M$
	\begin{align*}
		\Norm{S_N-S_M}{L^2(m(x)dx)}^2&=\sum_{n=M+1}^{N}\left(\int_0^{t}\frac{\fe_\Phi(t+h-\tau,\cdot)-\fe_\Phi(t-\tau,\cdot)}{h}\overline{\nu}_\Phi(\tau)d\tau+\lambda_n \fe_\Phi(t;-\lambda_n)\right)^2g^2_n\\
		&\le 2\sum_{n=M+1}^{N}\left(\left(\int_0^{t}\frac{\fe_\Phi(t+h-\tau,\cdot)-\fe_\Phi(t-\tau,\cdot)}{h}\overline{\nu}_\Phi(\tau)d\tau\right)^2+\left(\lambda_n \fe_\Phi(t;-\lambda_n)\right)^2\right)g^2_n\\
		&\le 2\left(\frac{1}{h^2}\cI^2_\Phi(t)+K^2(t)\right)\sum_{n=M+1}^{N}g_n^2,
	\end{align*}
where we used Equation \eqref{unifest}. Being $(g_n)_{n \ge 0}\in \ell^2$, we have that $S_N$ converges in $L^2(m(x)dx)$ by Cauchy's criterion. This implies that we can use Parseval's identity in the right-hand side of \eqref{limpass23} to achieve
	\begin{align}\label{limpass24}
	\begin{split}
		\Norm{\vphantom{\frac{1}{h}}\right.&\left.\sum_{n=1}^{+\infty}\left(\int_0^{t}\frac{\fe_\Phi(t+h-\tau,\cdot)-\fe_\Phi(t-\tau,\cdot)}{h}\overline{\nu}_\Phi(\tau)d\tau+\lambda_n \fe_\Phi(t;-\lambda_n)\right)g_nQ_n}{L^2(m(x)dx)}^2\\
		&=\sum_{n=1}^{+\infty}\left(\int_0^{t}\frac{\fe_\Phi(t+h-\tau,\cdot)-\fe_\Phi(t-\tau,\cdot)}{h}\overline{\nu}_\Phi(\tau)d\tau+\lambda_n \fe_\Phi(t;-\lambda_n)\right)^2g_n^2\\
		&=\sum_{n=1}^{+\infty}\left(\frac{J_\Phi(t+h)-J_\Phi(t)}{h}-\frac{1}{h}\int_{t}^{t+h}(\fe_\Phi(t+h-\tau,\cdot)-\fe_\Phi(0,\cdot))\overline{\nu}_\Phi(\tau)d\tau+\lambda_n \fe_\Phi(t;-\lambda_n)\right)^2g_n^2,
	\end{split}
\end{align}
where
\begin{equation*}
	J_\Phi(t)=\int_0^t \fe_\Phi(t-\tau;-\lambda_n)\overline{\nu}_\Phi(\tau)d\tau.
\end{equation*}
Now let us observe that $J_\Phi$ belongs to $C^1$ and
\begin{equation*}
	J_\Phi'(t)=-\lambda_n \fe_\Phi(t;-\lambda_n)
\end{equation*}
thus, by Lagrange's theorem, we know there exists $\theta_n(h)\in [t,t+h]$ such that
\begin{equation*}
	\frac{J_\Phi(t+h)-J_\Phi(t)}{h}=-\lambda_n \fe_\Phi(\theta_n(h);-\lambda_n).
\end{equation*}
Using the latter in Equation \eqref{limpass24} we obtain
	\begin{align}\label{limpass25}
	\begin{split}
		\Norm{\vphantom{\frac{1}{h}}\right.&\left.\sum_{n=1}^{+\infty}\left(\int_0^{t}\frac{\fe_\Phi(t+h-\tau,\cdot)-\fe_\Phi(t-\tau,\cdot)}{h}\overline{\nu}_\Phi(\tau)d\tau+\lambda_n \fe_\Phi(t;-\lambda_n)\right)g_nQ_n}{L^2(m(x)dx)}^2\\
		&\le 2\sum_{n=1}^{+\infty}\left(\vphantom{\frac{1}{h}}\lambda_n^2(\fe_\Phi(t;-\lambda_n)-\fe_\Phi(\theta_n(h);-\lambda_n))^2+\left(\frac{\overline{\nu}_\Phi(t)}{h}\int_{0}^{h}(\fe_\Phi(\tau,\cdot)\vphantom{\frac{1}{h}}-\fe_\Phi(0,\cdot))d\tau\right)^2\right)g_n^2.
	\end{split}
\end{align}
Now let us show that the series in the right-hand side of \eqref{limpass25} normally converges. Indeed, considering $h \in [0,\delta]$ for some $\delta>0$, we have
\begin{align*}
	\sum_{n=1}^{+\infty}&\left(\vphantom{\frac{1}{h}}\lambda_n^2(\fe_\Phi(t;-\lambda_n)-\fe_\Phi(\theta_n(h);-\lambda_n))^2+\left(\frac{\overline{\nu}_\Phi(t)}{h}\int_{0}^{h}(\fe_\Phi(\tau,\cdot)\vphantom{\frac{1}{h}}-\fe_\Phi(0,\cdot))d\tau\right)^2\right)g_n^2\\
	&\le (2K^2(t+\delta)+\overline{\nu}_\Phi^2(t))\sum_{n=1}^{+\infty}g_n^2.
\end{align*}
Thus we can take the limit as $h \to 0^+$ inside the summation sign, obtaining, since $\fe_\Phi(t;-\lambda_n)$ is continuous,
\begin{align*}
	\begin{split}
		\lim_{h \to 0^+}\Norm{\vphantom{\frac{1}{h}}\sum_{n=1}^{+\infty}\left(\int_0^{t}\frac{\fe_\Phi(t+h-\tau,\cdot)-\fe_\Phi(t-\tau,\cdot)}{h}\overline{\nu}_\Phi(\tau)d\tau+\lambda_n \fe_\Phi(t;-\lambda_n)\right)g_nQ_n}{L^2(m(x)dx)}=0,
	\end{split}
\end{align*}
that is to say
\begin{equation}\label{limpass2}
	\lim_{h \to 0^+}\Norm{\frac{1}{h}\int_0^{t}(u(t+h-\tau,\cdot)-u(t-\tau,\cdot))\overline{\nu}_\Phi(\tau)d\tau+\sum_{n=1}^{+\infty}\lambda_n \fe_\Phi(t;-\lambda_n)g_nQ_n}{L^2(m(x)dx)}=0.
\end{equation}
Taking the limit as $h \to 0^+$ in \eqref{limpass01} and using Equations \eqref{limpass1} and \eqref{limpass2} we achieve Equation \eqref{limpass0}.\\
Uniqueness of the solution follows from the fact that $(Q_n)_{n \in \N}$ is a complete orthonormal system in $L^2(m(x)dx)$. Finally, the fact that $p_\Phi(t,x;y)$ is the fundamental solution of \cref{nonlocback} follows from the spectral decomposition \cref{thm:specI}.
\end{proof}
\begin{rmk}
	Let us observe that if $g \in C^2_0(E)$ and $\Phi$ is a complete Bernstein function, we can use \cref{thmsemig} to obtain that $u(t,y)$ defined in Equation \cref{stocrepI} is solution of \cref{nonlocback} and can be extended to an analytic function in a sector $\C(\alpha)$ for some $\alpha<\frac{\pi}{2}$. 
\end{rmk}
\subsection{The non-local forward Kolmogorov equation}
Here we want to focus on the following non-local Cauchy problem
\begin{equation}\label{nonlocfor}
\begin{cases}
\partial_t^\Phi v(t,x)=\cF v(t,x) & t>0, \ x \in E \\
v(0,x)=f(x) & x \in E
\end{cases}
\end{equation}
for suitable initial datum $f$. For the definition of strong solution, we still refer to \cref{strongsol}. To do this, we first need to show the following preliminary result, whose proof can be omitted since it follows by direct calculations and by using Pearson equation \cref{Peareq}.
\begin{lemma}
	Let $X(t)$ be a Pearson diffusion of spectral category I with diffusion space $E$, associated family of orthonormal polynomials $Q_n$ with respective eigenvalues $\lambda_n$, stationary density $m(x)$ and Fokker-Planck operator $\cF$. Then
	\begin{equation*}
	\cF(m(x)Q_n(x))=-\lambda_nm(x)Q_n(x).
	\end{equation*}
\end{lemma}

Now that we have proven this Lemma, we can actually show existence, uniqueness and spectral decomposition of the strong solutions (in $L^2(m(x)dx)$) of \cref{nonlocfor}.
\begin{theorem}\label{thm:strongsolfor}
	Let $X(t)$ be a Pearson diffusion of spectral category I with diffusion space $E$, associated family of orthonormal polynomials $Q_n$ with respective eigenvalues $\lambda_n$, Fokker-Planck operator $\cF$ and bounded stationary density $m(x)$. Let $\Phi \in \BF$ be a driftless Bernstein function. Let $f:E \to \R$ be such that $f/m \in L^2(m(x)dx)$ is decomposed as
	\begin{equation*}
	f/m=\sum_{n=0}^{+\infty}f_nQ_n
	\end{equation*}
	where the series converges in $L^2(m(x)dx)$, absolutely for fixed $x \in E$ and uniformly on compact intervals $[x_1,x_2]\subset E$. Then the unique strong solution in $L^2(m(x)dx)$ of the problem \cref{nonlocback} is given by
	\begin{equation}\label{forsol1}
	v(t,x)=m(x)\sum_{n=0}^{+\infty}\fe_\Phi(t;-\lambda_n)Q_n(x)f_n.
	\end{equation}
	Moreover, $p_\Phi(t,x;y)$ is the fundamental solution of the problem \cref{nonlocback}, in the sense that, for any initial datum $f$ such that $f/m\in L^2(m(x)dx)$ and $f/m=\sum_{n=0}^{+\infty}f_nQ_n$ converges in $L^2(m(x)dx)$, absolutely for fixed $x \in E$ and uniformly on compact interval $[x_1,x_2]\subset E$, it holds
	\begin{equation*}
	v(t,x)=\int_{E}p_\Phi(t,x;y)f(y)dy.
	\end{equation*}
\end{theorem}
\begin{proof}
	Since we are supposing that $m(x)$ is bounded, it is easy to see that if a function $f$ is such that $f/m \in L^2(m(x)dx)$, then $f \in L^2(m(x)dx)$. Thus the convergence of the series in \cref{forsol1} in $L^2(m(x)dx)$ is ensured by the convergence of the series defined by $v(t,x)/m(x)$. Moreover, the strong continuity of $v(t,x)$ follows from the one of $v(t,x)/m(x)$. For the single summand of \cref{forsol1}, it holds, by using the previous Lemma,
	\begin{equation*}
	\partial_t^\Phi m(x)\fe_\Phi(t;-\lambda_n)Q_n(x)f_n=\cF m(x)\fe_\Phi(t;-\lambda_n)Q_n(x)f_n.
	\end{equation*}
	Finally, the exact same proof of \cref{thm:strongsolback} still works in this case.
\end{proof}
\begin{rmk}
	If $m$ is not bounded (i.e. in the cases of the CIR process as $b \in (0,1)$ and of the Jacobi process as $a,b \in (-1,0)$), one can prove in the same way that $v(t,x)$ is a strong solution in $L^2(m^{-1}(x)dx)$.\\
	If, additionally, $f \in L^2(m(x)dx)$ and the same convergence assumptions on $f/m$ hold also for $f=\sum_{n=0}^{+\infty}Q_n\widetilde{f}_n$, then $v(t,x)$ is a strong solution in $L^2(m(x)dx)$ even if $m(x)$ is unbounded. Finally, let us recall that if $f \in C^2_0(E)$ and $\Phi$ is a complete Bernstein function, then, by \cref{lem:adj} and \cref{thmsemig}, we know that $v(t,\cdot)$ can be analytically extended in a sectorial region $\C(\alpha)$ where $\alpha<\frac{\pi}{2}$.
\end{rmk}
\section{Spectral decomposition of non-local Pearson diffusions of spectral category II and III}\label{Sec6}
Now we consider the case of Pearson diffusions of spectral category II and III. We discuss them together since to exploit strong solutions of Equations \cref{nonlocback} and \cref{nonlocfor} we have to directly use semigroup theory, being the spectral decomposition of $p_\Phi(t,x;y)$ more complicated.\\
Let us first focus on the spectral decomposition of the transition densities.
\subsection{Spectral decomposition of the transition density for spectral category II}
Now let us show a spectral decomposition theorem for the transition density $p_\Phi(t,x;y)$ of a non-local Pearson diffusion of spectral category II.
\begin{theorem}\label{thm:specdecII}
	Let $\Phi \in \BF$ be a driftless Bernstein function, $X(t)$ be a Pearson diffusion of spectral category II with diffusion space $E$, associated family of classical orthonormal polynomials $Q_n(x)$ for $n \le N_j$, with respective eigenvalues in the discrete spectrum of $-\cG$ given by $\lambda_n$ and $m(x)$ stationary density. Then, for any $t>0$ and $x,x_0 \in E$, it holds
	\begin{equation*}
	p_\Phi(t,x;x_0)=p_{d,\Phi}(t,x;x_0)+p_{c,\Phi}(t,x;x_0)
	\end{equation*} 
	where
	\begin{equation}\label{pdPhi1}
	p_{d,\Phi}(t,x;x_0)=m(x)\sum_{n=0}^{N_j}\fe_\Phi(t;-\lambda_n)Q_n(x_0)Q_n(x)
	\end{equation}
	and
	\begin{equation*}
	p_{c,\Phi}(t,x;x_0)=\frac{m(x)}{\pi}\int_{\Lambda_j}^{+\infty}\fe_\Phi(t;-\lambda)a_j(\lambda)f_j(x;-\lambda)f_j(x_0;-\lambda)d\lambda,
	\end{equation*}
	where $j=1,2$ and the involved functions are defined in \cref{subsecspecII}.
\end{theorem}
\begin{proof}
	By \cref{spedecRG} we already know that
	\begin{equation*}
	p(t,x;x_0)=p_{d}(t,x;x_0)+p_{c}(t,x;x_0)
	\end{equation*} 
	where
	\begin{equation*}
	p_{d}(t,x;x_0)=m(x)\sum_{n=0}^{N_j}e^{-\lambda_n t}Q_n(x_0)Q_n(x)
	\end{equation*}
	and
	\begin{equation*}
	p_{c}(t,x;x_0)=\frac{m(x)}{\pi}\int_{\Lambda_j}^{+\infty}e^{-\lambda t}a_j(\lambda)f_j(x;-\lambda)f_j(x_0;-\lambda)d\lambda.
	\end{equation*}
	Moreover, by Equation \cref{intrep} we know that
	\begin{align*}
	p_\Phi(t,x;x_0)&=\int_0^{+\infty}p(s,x;x_0)f_\Phi(s;t)ds\\
	&=\int_0^{+\infty}p_d(s,x;x_0)f_\Phi(s;t)ds+\int_0^{+\infty}p_c(s,x;x_0)f_\Phi(s;t)ds\\
	&:=p_{d,\Phi}(t,x;x_0)+p_{c,\Phi}(t,x;x_0).
	\end{align*}
	Concerning $p_{d,\Phi}$, we already know that it is given by equation \cref{pdPhi1}, by linearity of the integral. Thus let us argue on $p_{c,\Phi}$. To use Fubini's theorem we only have to show that
	\begin{equation*}
	\int_{\Lambda_j}^{+\infty}\fe_\Phi(t;-\lambda)|a_j(\lambda)f_j(x;-\lambda)f_j(x_0;-\lambda)|d\lambda<+\infty.
	\end{equation*}
	The estimates in \cite[Pages $3522$ and $3526$]{leonenko2017heavy} lead to (for fixed $x,x_0 \in E$)
	\begin{equation*}
	|a_j(\lambda)f_j(x;-\lambda)f_j(x_0;-\lambda)|\le C_j |\Delta_j(\lambda)|^{-1}(1+O(|\Delta_j(\lambda)|^{-1})).
	\end{equation*}
	Moreover, we have that 
	\begin{equation*}
	\fe_\Phi(t;-\lambda)\le M(t)\lambda^{-1}
	\end{equation*}
	and $|\Delta_j(\lambda)| \ge C_j \lambda^{\frac{1}{2}}$ as $\lambda \to +\infty$. Thus we finally get
	\begin{equation*}
	\fe_\Phi(t;-\lambda)|a_j(\lambda)f_j(x;-\lambda)f_j(x_0;-\lambda)|=O(\lambda^{-\frac{3}{2}}),
	\end{equation*}
	concluding the proof.
\end{proof}
\subsection{Spectral decomposition of the transition density for spectral category III}
Concerning the spectral decomposition of $p_\Phi(t,x;y)$ for spectral category III, we will not use asymptotics of the eigenfunctions. Instead, we will use the fact that the continuous part of the spectrum is of multiplicity two on our advantage.
\begin{theorem}\label{thm:Stud}
	Let $\Phi \in \BF$ be a driftless Bernstein function, $X(t)$ be a Student process with diffusion space $E=\R$, associated family of classical orthonormal polynomials $Q_n(x)$ for $n \le N_2$ with respective eigenvalues in the discrete spectrum of $-\cG$ given by $\lambda_n$ and $m(x)$ Student density. Then, for any $t>0$ and $x,x_0 \in E$, it holds
	\begin{equation*}
	p_\Phi(t,x;x_0)=p_{d,\Phi}(t,x;x_0)+p_{c,\Phi}(t,x;x_0)
	\end{equation*} 
	where
	\begin{equation}\label{pdPhi}
	p_{d,\Phi}(t,x;x_0)=m(x)\sum_{n=0}^{N_3}\fe_\Phi(t;-\lambda_n)Q_n(x_0)Q_n(x)
	\end{equation}
	and
	\begin{multline}\label{pcPhi}
	p_{c,\Phi}(t,x;x_0)=m(x)\int_{\Lambda_3}^{+\infty}\fe_\Phi(t;-\lambda)\left(\frac{f_3(x_0,-\lambda)f_3(x,-\lambda)}{\Norm{f_j(\cdot,-\lambda)}{L^2(m(dx))}^2}+\frac{\bar{f}_3(x_0,-\lambda)\bar{f}_3(x,-\lambda)}{\Norm{\bar{f}_j(\cdot,-\lambda)}{L^2(m(dx))}^2}\right.\\\left.+\frac{f_3(x_0,-\lambda)\bar{f}_3(x,-\lambda)+\bar{f}_3(x_0,-\lambda)f_3(x,-\lambda)}{\Norm{\bar{f}_3(\cdot,-\lambda)}{L^2(m(dx))}\Norm{f_3(\cdot,-\lambda)}{L^2(m(dx))}}\right)d\lambda.
	\end{multline}
\end{theorem}
\begin{proof}
	As in \cref{thm:specdecII}, we only need to deal with $p_{c,\Phi}$. Let us denote 
	\begin{equation*}
	u_3(x,-\lambda)=\Re\left(\frac{f_3(x,-\lambda)}{\Norm{f_3(x,-\lambda)}{L^2(m(dx))}}\right).
	\end{equation*}
	Then it is easy to see, by simple complex analysis arguments, that
	\begin{equation*}
	p_c(t,x;x_0)=4m(x)\int_{\Lambda_3}^{+\infty}e^{-\lambda t}u_3(x,-\lambda)u_3(x_0,-\lambda)d\lambda.
	\end{equation*}
	Now set $x=x_0$ and let us first show the spectral decomposition in this case. 
	Observe that
	\begin{equation*}
	\int_{0}^{+\infty}\int_{\Lambda_3}^{+\infty}e^{-\lambda s}u^2_3(x,-\lambda)d\lambda f_\Phi(s;t)ds=\int_{\Lambda_3}^{+\infty}\fe_\Phi(t;-\lambda)u_3^2(x,-\lambda)d\lambda
	\end{equation*}
	by Fubini-Tonelli theorem, since the integrands are all non-negative. Thus we obtain, by Equation \cref{intrep},
	\begin{equation}\label{specdeceq3}
	p_\Phi(t,x;x)=p_{d,\Phi}(t,x;x)+p_{c,\Phi}(t,x;x).
	\end{equation} 
	Now let us fix $x \in \R$ and consider any scale function $\fs:\R \to \R$, i.e. a continuous strictly increasing function on $\R$ such that for any $a,b \in \R$ with $a<b$ and $y \in (a,b)$ it holds
	\begin{equation*}
	\bP_y(T_b<T_a)=\frac{\fs(y)-\fs(a)}{\fs(b)-\fs(a)}
	\end{equation*}
	where $T_b=\inf\{t>0: \ X(t)>b\}$ and $T_a=\inf\{t>0: \ X(t)<b\}$ (see \cite[Proposition $3.2$, Chapter $VII$]{revuz2013continuous}). Let us denote by $\fs^{-1}$ its inverse function. Thus, define
	\begin{equation*}
	V_x(t)=t\int_{\fs^{-1}(\fs(x)-t)}^{\fs^{-1}(\fs(x)+t)}sp(y)dy,
	\end{equation*}
	where $sp(y)$ is the speed density of $X(t)$. $V_x(t)$ is strictly increasing and continuous, thus so it is its inverse function $V^{-1}_x(t)$. Moreover, as $V_x(0)=0$ (by dominated convergence theorem, since the speed measure is finite) we also have $V_x^{-1}(0)=0$ and then $V_x^{-1}(t)$ is bounded in a neighbourhood of $0$. Then, by \cite[Theorem $2.1$]{tomisaki1977asymptotic}, we know that there exist two constants $t_0,C>0$ (depending on $x$) such that for any $t<t_0$ it holds
	\begin{equation*}
	\int_0^t p(t,x;x)dt\le C V^{-1}_x(t).
	\end{equation*}
	In particular this means that there exists $\varepsilon>0$ (for instance $\varepsilon=t_0/2$) such that $\int_0^{\varepsilon}p(t,x;x)dt<+\infty$.\\
	Now let us recall that $f_\Phi(s;t)$ is continuous for fixed $t$ and $f_\Phi(0^+;t)=\bar{\nu}_\Phi(t)$. Thus, for $s \in [0,\varepsilon]$, there exists a constant $M(t)>0$ such that $f_\Phi(s;t)\le M(t)$. Moreover, from the spectral decomposition, we have that $p(t,x;x)$ is decreasing in $t$, thus $p(t,x;x)\le p(\varepsilon,x;x)$ for any $t \ge \varepsilon$. Thus we have
	\begin{align*}
	p_\Phi(t,x;x)&=\int_0^{+\infty}p(s,x;x)f_\Phi(s;t)ds\\
	&=\int_0^{\varepsilon}p(s,x;x)f_\Phi(s;t)ds+\int_\varepsilon^{+\infty}p(s,x;x)f_\Phi(s;t)ds\\
	&\le M(t)\int_0^{\varepsilon}p(s,x;x)ds+p(\varepsilon,x;x)\int_{\varepsilon}^{+\infty}f_\Phi(s;t)ds<+\infty.
	\end{align*}
	Thus, by Equation \cref{specdeceq3} and the fact that both the summands are non-negative, we have that
	\begin{equation*}
	\int_{\Lambda_3}^{+\infty}\fe_\Phi(t;-\lambda)u_3^2(x;-\lambda)d\lambda<+\infty.
	\end{equation*}
	Now we can consider the case $x_0 \not = x$. Indeed, in such case, we have
	\begin{multline*}
	\left(\int_{\Lambda_3}^{+\infty}\fe_\Phi(t;-\lambda)|u_3(x_0;-\lambda)u_3(x;-\lambda)|d\lambda\right)^2 \\\le \left(\int_{\Lambda_3}^{+\infty}\fe_\Phi(t;-\lambda)u_3^2(x;-\lambda)d\lambda\right)\left(\int_{\Lambda_3}^{+\infty}\fe_\Phi(t;-\lambda)u_3^2(x_0;-\lambda)d\lambda\right)<+\infty,
	\end{multline*}
	thus we can use Fubini's theorem also in the case $x \not = x_0$ to conclude the proof.
\end{proof}
\subsection{The non-local backward and forward Kolmogorov equations}
Now we can focus on the strong solutions of the non-local backward and forward Kolmogorov equations \cref{nonlocback} and \cref{nonlocfor}. Both the results are direct consequence of \cite[Theorem $2.1$]{chen2017time}.
\begin{theorem}
	Let $X(t)$ be a Pearson diffusion of spectral category II or III with diffusion space $E$ and generator $\cG$. Let $\Phi \in \BF$ be a driftless Bernstein function and $g \in C^2_0(E)$. Then the strong solution of \cref{nonlocback} with initial datum $g$ is given by
	\begin{equation*}
	u(t,y)=\int_{E}p_\Phi(t,x;y)g(x)dx.
	\end{equation*}	
\end{theorem}
\begin{theorem}
	Let $X(t)$ be a Pearson diffusion of spectral category II with diffusion space $E$ and Fokker-Planck operator $\cF$. Let $\Phi \in \BF$ be a driftless Bernstein function. Let $f \in C^2_c(E)$. Then the strong solution of \cref{nonlocfor} with initial datum $f$ is given by
	\begin{equation*}
	v(t,x)=\int_{E}p_\Phi(t,x;y)f(y)dy.
	\end{equation*}	
\end{theorem}
\begin{rmk}
	For both previous Theorems, if $\Phi$ is a complete Bernstein function, then $u(t,x)$ and $v(t,x)$ can be extended by analyticity to a sectorial region $\C(\alpha)$ with $\alpha<\frac{\pi}{2}$.
\end{rmk}
\section{Stochastic representation of the solutions of the backward and forward Kolmogorov equations}\label{Sec7}
Let us recall that in the previous two Sections we have characterized $p_\Phi(t,x;y)$ as the fundamental solution of both problems \cref{nonlocback} and \cref{nonlocfor}. Recalling that $p_\Phi(t,x;y)$, we obtain some stochastic representation results. Concerning the backward Kolmogorov equation \cref{nonlocback} we have the following result.
The fact that $p_\Phi(t,x;y)$ is the fundamental solutions leads to a stochastic representation result for strong solutions of \cref{nonlocback}.
\begin{proposition}\label{prop:stocrepback}
	Let $X(t)$ be a Pearson diffusion with diffusion space $E$, associated family of orthonormal polynomials $Q_n$ with respective eigenvalues $\lambda_n$, generator $\cG$ and stationary density $m(x)$. Let $\Phi \in \BF$ be a driftless Bernstein function and $g \in C^2_0(E)$. Then the unique strong solution of \cref{nonlocback} is given by
	\begin{equation}\label{stocrepI}
	u(t,y)=\E_y[g(X_\Phi(t))].
	\end{equation}
\end{proposition}
\begin{proof}
	By definition of conditional expectation and the fact that $p_\Phi(t,x;y)$ is the transition density of $X_\Phi(t)$, we have
	\begin{equation*}
	u(t,y)=\int_{E}p_\Phi(t,x;y)g(x)dx
	\end{equation*}
	concluding the proof by observing that $p_\Phi(t,x;y)$ is the fundamental solution of \cref{nonlocback}.
\end{proof}
\begin{rmk}
	If $X_\Phi(t)$ belongs to spectral category I, then one can weaken the hypotheses on the initial datum, by asking the same hypotheses as in \cref{thm:strongsolback} in place of $g \in C^2_0(E)$.
\end{rmk}
However, since we are using adjoint operators, we have to ask for more strict hypotheses for the stochastic representation of strong solutions of \cref{nonlocfor}.
\begin{proposition}\label{prop:forsolstocrep}
	Let $X(t)$ be a Pearson diffusion with diffusion space $E$, associated family of orthonormal polynomials $Q_n$ with respective eigenvalues $\lambda_n$, Fokker-Planck operator $\cF$ and stationary density $m(x)$. Let $\Phi \in \BF$ be a driftless Bernstein function and $f \in L^1(dx) \cap C^2_c(E)$ with $f \ge 0$ and $\Norm{f}{L^1(dx)}=1$. Then the unique strong solution of the problem \cref{nonlocfor} is given by
	\begin{equation}\label{forsol2}
	v(t,x)=\der{}{x}\bP_f(X_\Phi(t)\le x)
	\end{equation}
	where $\bP_f(\cdot)$ is the probability measure obtained by $\bP$ conditioning with the request that $X_\Phi(0)$ admits $f(x)dx$ as probability law.
\end{proposition}
\begin{proof}
	Let us observe that (setting $p_\Phi(t,x;y)$ as $p_\Phi(t,x;y)\chi_{E}(x)$)
	\begin{align*}
	\bP_f(X_\Phi(t)\le x)&=\int_0^{+\infty}\bP_y(X_\Phi(t)\le x)f(y)dy\\
	&=\int_0^{+\infty}\int_{-\infty}^{x}p_\Phi(t,z;y)dzf(y)dy\\
	&=\int_{-\infty}^{x}\int_0^{+\infty}p_\Phi(t,z;y)f(y)dydz.
	\end{align*}
	Thus it holds
	\begin{equation*}
	v(t,x)=\der{}{x}\bP_f(X_\Phi(t)\le x)=\int_0^{+\infty}p_\Phi(t,x;y)f(y)dy
	\end{equation*}
	concluding the proof by the fact that $p_\Phi(t,x;y)$ is the fundamental solution of \cref{nonlocfor}.
\end{proof}
\begin{rmk}
	As for \cref{prop:stocrepback}, if $X_\Phi(t)$ belongs to spectral category I, then one can weaken the hypotheses on the initial datum, by asking the same hypotheses as in \cref{thm:strongsolfor} in place of $f \in C^2_c(E)$. 
\end{rmk}
\section{Limit distributions, first order stationarity and correlation structure of non-local Pearson diffusions}\label{Sec8}
Now let us use the previous results to establish some properties of non-local Pearson diffusions $X_\Phi(t)$. Let us first show that for regular enough initial distributions the limit distribution of any non-local Pearson diffusion admits is given by the stationary distribution $m(x)dx$ of the respective Pearson diffusion.
\begin{theorem}
	Let $X(t)$ be a Pearson diffusion with diffusion space $E$, associated family of orthonormal polynomials $Q_n$ with respective eigenvalues $\lambda_n$ and stationary distribution $m(x)dx$. Let $\Phi \in \BF$ be a driftless Bernstein function. Then
	\begin{equation}\label{eq:limit}
	\lim_{t \to +\infty}p_\Phi(t,x;y)\to m(x)
	\end{equation}
	for fixed $x,y \in E$.\\
	Moreover, for any $f \in C^2_c(E) \cap L^1(E)$ such that $f \ge 0$ and $\Norm{f}{L^1(E)}=1$, it holds
	\begin{equation*}
	\lim_{t \to +\infty}\left(\der{}{x}\bP_f(X_\Phi(t)\le x)\right)=m(x)
	\end{equation*}
	for any fixed $x \in E$.
\end{theorem}
\begin{proof}
	Let us first argue for $X_\Phi(t)$ of spectral category I. First of all, we have (since for any Pearson diffusion $\lambda_0=0$ and $Q_0(x)\equiv 1$):
	\begin{equation*}
	p_\Phi(t,x;y)=m(x)+\sum_{n=1}^{+\infty}\fe_\Phi(t;-\lambda_n)Q_n(x)Q_n(y).
	\end{equation*}
	Since we are sending $t \to +\infty$, we can suppose $t\ge t_0>0$ and then $\fe_\Phi(t;-\lambda_n)\le \fe_\Phi(t_0;-\lambda_n)$. By \cref{lem:contrseries}, we can use dominated convergence theorem to obtain Equation \cref{eq:limit}.\\
	Concerning $X_\Phi(t)$ of spectral category II, arguing as before, we have (for $j=1,2$)
	\begin{equation*}
	p_\Phi(t,x;y)=m(x)+\sum_{n=1}^{N_j}\fe_\Phi(t;-\lambda_n)Q_n(x)Q_n(y)+m(x)\int_{\Lambda_j}^{+\infty}\fe_\Phi(t;-\lambda)a_j(\lambda)f_j(x;-\lambda)f_j(y;-\lambda)d\lambda,
	\end{equation*}
	where $\sum_{n=1}^{0}\equiv 0$. Arguing as before and observing as in \cref{thm:specdecII} that
	\begin{equation*}
 |\fe_\Phi(t_0;-\lambda)a_j(\lambda)f_j(x;-\lambda)f_j(y;-\lambda)|=O(\lambda^{-\frac{3}{2}}),
	\end{equation*}
	we can use again dominated convergence theorem to achieve Equation \cref{eq:limit}.\\
	Finally, if $X_\Phi(t)$ is a Student diffusion we have
	\begin{equation*}
	p_\Phi(t,x;y)=m(x)+\sum_{n=1}^{N_3}\fe_\Phi(t;-\lambda_n)Q_n(x)Q_n(y)+4m(x)\int_{\Lambda_3}^{+\infty}\fe_\Phi(t;-\lambda)u_3(x;-\lambda)u_3(y;-\lambda)d\lambda,
	\end{equation*}
	where $u_3$ is defined in the proof of \cref{thm:Stud}. Since we have shown by Cauchy-Schwartz inequality that $|\fe_\Phi(t_0;-\lambda)u_j(x;-\lambda)u_j(y;-\lambda)|$ is integrable in $\lambda$, we can still use dominated convergence theorem to achieve Equation \cref{eq:limit}.\\
	Finally, arguing as in \cref{prop:forsolstocrep}, we have
	\begin{equation*}
	\der{}{x}\bP_f(X_\Phi(t)\le x)=\int_0^{+\infty}p_\Phi(t,x;y)f(y)dy
	\end{equation*}
	and then we conclude the proof by dominated convergence theorem.
\end{proof}
\begin{rmk}
	If $X_\Phi(t)$ belongs to spectral category I, then we can substitute the hypothesis $f \in C^2_c(E)$ with the ones of \cref{thm:strongsolfor}.
\end{rmk}
In particular, we can show that $m(x)$ is actually the first-order stationary distribution of $X_\Phi(t)$.
\begin{proposition}
	Let $X(t)$ be a Pearson diffusion with diffusion space $E$, associated family of orthonormal polynomials $Q_n$ with respective eigenvalues $\lambda_n$ and stationary distribution $m(x)dx$. Let $\Phi \in \BF$ be a driftless Bernstein function.
	Then it holds
	\begin{equation*}
	\der{}{x}\bP_m(X_\Phi(t)\le x)=m(x)
	\end{equation*}
	for any $t \ge 0$ and $x \in E$.
\end{proposition}
\begin{proof}
	Let us observe that
	\begin{align*}
	\der{}{x}\bP_m(X_\Phi(t)\le x)&=\int_{E}p_\Phi(t,x;y)m(y)dy\\
	&=\int_{E}\int_{0}^{+\infty}p(s,x;y)f_\Phi(s;t)dsm(y)dy\\
	&=\int_{0}^{+\infty}\int_{E}p(s,x;y)m(y)dyf_\Phi(s;t)ds\\
	&=m(x)\int_{0}^{+\infty}f_\Phi(s;t)ds=m(x)
	\end{align*}
	where we used Tonelli-Fubini's theorem (since the integrand is non-negative) and the fact that $m(x)$ is the stationary measure of $X(t)$.
\end{proof}
\begin{rmk}
	An interesting problem we leave open consists in giving estimates on the convergence rates of non-stationary non-local Pearson diffusions, as done in the classical case for the Fisher-Snedecor diffusion in \cite{kulik2013ergodicity}. 
\end{rmk}
However, the process $X_\Phi(t)$ is not second-order stationary, neither in wide sense, as we can see from the following result.
\begin{proposition}\label{prop:cov}
	Let $X(t)$ be a Pearson diffusion with diffusion space $E$, associated family of orthonormal polynomials $Q_n$ with respective eigenvalues $\lambda_n$ and stationary distribution $m(x)$. Let $\Phi \in \BF$ be a driftless Bernstein function. Then it holds, for any $t \ge s \ge 0$,
	\begin{equation*}
	\Corr_m(X_\Phi(t),X_\Phi(s))=\lambda_1\int_0^{s}\fe_\Phi(t-\tau;-\lambda_1)dU_\Phi(\tau)-2+2\fe_\Phi(s;-\lambda_1)+\fe_\Phi(t;-\lambda_1),
	\end{equation*}
	where $U_\Phi(t)=\E[L_\Phi(t)]$.
\end{proposition}
\begin{proof}
	This is a consequence of the fact that $\Corr_m(X(t),X(s))=e^{-\lambda_1|t-s|}$ and \cite[Theorem $2$]{ascione2019semi}.
\end{proof}
Since, if $X_\Phi(0)$ admits $m(x)$ as probability density function, the process $X_\Phi(t)$ is first-order stationary, but not second-order stationary, we cannot exploit long-range or short-range dependence properties with the usual definitions. Thus, taking inspiration from \cite[Lemmas $2.1$ and $2.2$]{beran2016long}, we give the following definitions.
\begin{definition}
	Given the function $\gamma(n)=\Corr_m(X_\Phi(n),X_\Phi(0))$ for $n \in \N$:
	\begin{itemize}
		\item $X_\Phi(t)$ is said to exhibit long-range dependence if $\gamma(n)\sim \ell(n)n^{-\alpha}$ where $\ell(n)$ is a slowly varying function and $\alpha \in \left(0,1\right)$;
		\item $X_\Phi(t)$ is said to exhibit short-range dependence if $\sum_{n=1}^{+\infty}|\gamma(n)|<+\infty$.
	\end{itemize}
\end{definition}
Now let us observe that, by \cref{prop:cov} we know that
\begin{equation*}
\Corr_m(X_\Phi(t),X_\Phi(0))=\fe_\Phi(t;-\lambda_1)
\end{equation*}
thus we only have to exploit some asymptotic properties of $\fe_\Phi(t;-\lambda_1)$.
\begin{proposition}\label{prop:regvar}
	Let $\Phi \in \BF$ be a driftless Bernstein function. Then the following properties hold:
	\begin{enumerate}
		\item If $\Phi$ is regularly varying at $0^+$ with order $\alpha \in (0,1)$, then, for any fixed $\lambda>0$, $\fe_\Phi(t;-\lambda)$ is regularly varying at $\infty$ with order $-\alpha$.
		\item Suppose $\lim_{z \to 0^+}\frac{\Phi(z)}{z}=l \in (0,+\infty)$. Then, for fixed $\lambda>0$, $\fe_\Phi(t;-\lambda)$ is integrable in $(0,+\infty)$.
	\end{enumerate}
\end{proposition}
\begin{proof}
	Property $(1)$ has been already proved in \cite{ascione2020non} thus we focus on property $(2)$. Let us define the integral function $J(t)=\int_0^t \fe_\Phi(s;-\lambda)ds$ and set $\overline{J}$ the Laplace-Stieltjes transform of $J$. We have
	\begin{equation*}
	\overline{J}(z)=\frac{\Phi(z)}{z(\Phi(z)+\lambda)}.
	\end{equation*}
	Taking the limit as $z \to 0^+$, since $\Phi(0)=0$, we have $\lim_{z \to 0^+}\overline{J}(z)=l/\lambda$. Thus, by Karamata's Tauberian theorem, we conclude the proof.
\end{proof}
As a direct Corollary we have
\begin{corollary}\label{cor:longrange}
	Let $\Phi \in \BF$ be a driftless Bernstein function and $X_\Phi(t)$ be a non-local Pearson diffusion such that $X_\Phi(0)$ admits probability density function $m(x)$. Then the following properties hold:
	\begin{enumerate}
		\item[$(i)$] If $\Phi$ is regularly varying at $0^+$ with order $\alpha \in (0,1)$, then $X_\Phi(t)$ is long-range dependent with respect to the initial datum;
		\item[$(ii)$] If $\lim_{z \to 0^+}\frac{\Phi(z)}{z}=l \in (0,+\infty)$, then $X_\Phi(t)$ is short-range dependent with respect to the initial datum.
	\end{enumerate}
\end{corollary}
\begin{proof}
	Property $(i)$ directly follows from property $(1)$ of the previous proposition. Property $(ii)$ instead follows from property $(2)$ of the previous proposition and the integral criterion for convergence of the series.
\end{proof}
Now we can consider the examples we stated in \cref{Sec3}:
\begin{itemize}
	\item If $\Phi(\lambda)=\lambda^\alpha$ with $\alpha \in (0,1)$ then we have that $X_\Phi(t)$ is actually a fractional Pearson diffusion as exploited in \cite{leonenko2013fractional} for the first spectral category or \cite{leonenko2017heavy} for the second spectral category. In any case, $X_\alpha$ is long-range dependent with respect to the initial datum as $\Phi$ is regularly varying at $0^+$ with index $\alpha \in (0,1)$. In particular, by using the formula for the autocorrelation function given in \cite{leonenko2013correlation}, that is
	\begin{equation*}
	\Corr(X_\Phi(t),X_\Phi(s))=E_\alpha(-\lambda_1 t^\alpha)+\frac{\lambda_1 \alpha t^\alpha}{\Gamma(1+\alpha)}\int_0^{\frac{s}{t}}\frac{E_\alpha(-\lambda_1 t^\alpha(1-z)^\alpha)}{z^{1-\alpha}}dz,
	\end{equation*}
	we have that the process is actually long-range dependent with respect to any datum $X_\Phi(s)$.
	\item If $\Phi(\lambda)=(\lambda+\theta)^\alpha-\theta^\alpha$ for $\alpha \in (0,1)$, we have $\lim_{\lambda \to 0^+}\frac{\Phi(\lambda)}{\lambda}=\alpha$. Thus the tempered fractional Pearson diffusions $X_\Phi(t)$ are short-range dependent with respect to the initial datum;
	\item If $\Phi(\lambda)=\log(1+\lambda^\alpha)$ for $\alpha \in (0,1)$, we have $\lim_{\lambda \to 0^+}\frac{\Phi(\lambda)}{\lambda^\alpha}=1$, thus $\Phi(\lambda)$ is regularly varying at $0^+$ with index $\alpha \in (0,1)$. This implies that the geometric fractional Pearson diffusions $X_{\Phi}(t)$ are long-range dependent with respect to the initial datum;
	\item If $\Phi(\lambda)=\log(1+\lambda)$, we have $\lim_{\lambda \to 0^+}\frac{\Phi(\lambda)}{\lambda}=1$. This implies that the Gamma time-changed Pearson diffusions $X_{\Phi}(t)$ are short-range dependent with respect to the initial datum.
\end{itemize}


\bibliographystyle{siamplain}
\bibliography{bib}

\begin{thebibliography}{10}

\bibitem{abramowitz1988handbook}
{\sc M.~Abramowitz, I.~A. Stegun, and R.~H. Romer}, {\em Handbook of
  Mathematical Functions with Formulas, Graphs, and Mathematical Tables},
  American Association of Physics Teachers, 1988.

\bibitem{alrawashdeh2017applications}
{\sc M.~S. Alrawashdeh, J.~F. Kelly, M.~M. Meerschaert, and H.-P. Scheffler},
  {\em Applications of inverse tempered stable subordinators}, Computers \&
  Mathematics with Applications, 73 (2017), pp.~892--905.

\bibitem{amrein2005sturm}
{\sc W.~O. Amrein, A.~M. Hinz, and D.~B. Pearson}, {\em Sturm-Liouville Theory:
  Past and Present}, Springer Science \& Business Media, 2005.

\bibitem{arendtvector}
{\sc W.~Arendt, C.~J. Batty, M.~Hieber, and F.~Neubrander}, {\em Vector-valued
  Laplace Transforms and Cauchy Problems}, Springer, 2011.

\bibitem{arista2020explicit}
{\sc J.~Arista and N.~Demni}, {\em Explicit expressions of the {H}ua-{P}ickrell
  semi-group}, arXiv preprint arXiv:2008.07195,  (2020).

\bibitem{ascione2020generalized}
{\sc G.~Ascione}, {\em Abstract {C}auchy problems for the generalized
  fractional calculus}, arXiv preprint arXiv:2006.09789,  (2020).

\bibitem{ascione2019fractional}
{\sc G.~Ascione, N.~Leonenko, and E.~Pirozzi}, {\em Fractional
  immigration-death processes}, arXiv preprint arXiv:1907.07588,  (2019).

\bibitem{ascione2020non}
{\sc G.~Ascione, N.~Leonenko, and E.~Pirozzi}, {\em Non-local solvable
  birth-death processes}, arXiv preprint arXiv:2007.13656,  (2020).

\bibitem{ascione2019semi}
{\sc G.~Ascione and B.~Toaldo}, {\em A semi-{M}arkov leaky integrate-and-fire
  model}, Mathematics, 7 (2019), p.~1022.

\bibitem{avram2013spectral}
{\sc F.~Avram, N.~N. Leonenko, and N.~{\v{S}}uvak}, {\em On spectral analysis
  of heavy-tailed {K}olmogorov-{P}earson diffusions}, Markov Processes and
  Related Fields, 19 (2013), pp.~249--298.

\bibitem{avram2013spectralfs}
{\sc F.~Avram, N.~N. Leonenko, and N.~{\v{S}}uvak}, {\em Spectral
  representation of transition density of {F}isher--{S}nedecor diffusion},
  Stochastics An International Journal of Probability and Stochastic Processes,
  85 (2013), pp.~346--369.

\bibitem{baeumer2003inversion}
{\sc B.~Baeumer}, {\em On the inversion of the convolution and {L}aplace
  transform}, Transactions of the American Mathematical Society, 355 (2003),
  pp.~1201--1212.

\bibitem{baeumer2001stochastic}
{\sc B.~Baeumer and M.~M. Meerschaert}, {\em Stochastic solutions for
  fractional {C}auchy problems}, Fractional Calculus and Applied Analysis, 4
  (2001), pp.~481--500.

\bibitem{beran2016long}
{\sc J.~Beran, Y.~Feng, S.~Ghosh, and R.~Kulik}, {\em Long-{M}emory
  {P}rocesses}, Springer, 2016.

\bibitem{bertoin1996levy}
{\sc J.~Bertoin}, {\em L{\'e}vy Processes}, vol.~121, Cambridge University
  Press, 1996.

\bibitem{bertoin1999subordinators}
{\sc J.~Bertoin}, {\em Subordinators: Examples and applications}, in Lectures
  on Probability Theory and Statistics, Springer, 1999, pp.~1--91.

\bibitem{bingham1971limit}
{\sc N.~H. Bingham}, {\em Limit theorems for occupation times of {M}arkov
  processes}, Zeitschrift f{\"u}r Wahrscheinlichkeitstheorie und verwandte
  Gebiete, 17 (1971), pp.~1--22.

\bibitem{bouchaud1990anomalous}
{\sc J.-P. Bouchaud and A.~Georges}, {\em Anomalous diffusion in disordered
  media: statistical mechanisms, models and physical applications}, Physics
  reports, 195 (1990), pp.~127--293.

\bibitem{cao2014tempered}
{\sc J.~Cao, C.~Li, and Y.~Chen}, {\em On tempered and substantial fractional
  calculus}, in 2014 IEEE/ASME 10th International Conference on Mechatronic and
  Embedded Systems and Applications (MESA), IEEE, 2014, pp.~1--6.

\bibitem{chen2017time}
{\sc Z.-Q. Chen}, {\em Time fractional equations and probabilistic
  representation}, Chaos, Solitons \& Fractals, 102 (2017), pp.~168--174.

\bibitem{cinlar}
{\sc E.~Cinlar}, {\em Markov additive processes and semi-regeneration},
  Discussion Papers 118, Northwestern University, Center for Mathematical
  Studies in Economics and Management Science, 1974,
  \url{https://EconPapers.repec.org/RePEc:nwu:cmsems:118}.

\bibitem{da2020green}
{\sc J.~L. da~Silva and Y.~Kondratiev}, {\em Green measures for time changed
  {M}arkov processes}, arXiv preprint arXiv:2008.03390,  (2020).

\bibitem{demni2009large}
{\sc N.~Demni and M.~Zani}, {\em Large deviations for statistics of the
  {J}acobi process}, Stochastic Processes and their Applications, 119 (2009),
  pp.~518--533.

\bibitem{dunford1958linear}
{\sc N.~Dunford and J.~T. Schwartz}, {\em Linear Operators Part I: General
  Theory}, vol.~243, Interscience publishers New York, 1958.

\bibitem{eliazar2013fractional}
{\sc I.~I. Eliazar and M.~F. Shlesinger}, {\em Fractional motions}, Physics
  Reports, 527 (2013), pp.~101--129.

\bibitem{forman2008pearson}
{\sc J.~L. Forman and M.~S{\o}rensen}, {\em The {P}earson diffusions: A class
  of statistically tractable diffusion processes}, Scandinavian Journal of
  Statistics, 35 (2008), pp.~438--465.

\bibitem{gajda2015time}
{\sc J.~Gajda and A.~Wy{\l}oma{\'n}ska}, {\em Time-changed
  {O}rnstein--{U}hlenbeck process}, Journal of Physics A: Mathematical and
  Theoretical, 48 (2015), p.~135004.

\bibitem{gorenflo2003fractional}
{\sc R.~Gorenflo and F.~Mainardi}, {\em Fractional diffusion processes:
  probability distributions and continuous time random walk}, in Processes with
  Long-Range Correlations, Springer, 2003, pp.~148--166.

\bibitem{henry2010introduction}
{\sc B.~I. Henry, T.~A. Langlands, and P.~Straka}, {\em An introduction to
  fractional diffusion}, in Complex Physical, Biophysical and Econophysical
  Systems, World Scientific, 2010, pp.~37--89.

\bibitem{hille1996functional}
{\sc E.~Hille and R.~S. Phillips}, {\em Functional Analysis and Semi-groups},
  vol.~31, American Mathematical Soc., 1996.

\bibitem{ismail2005classical}
{\sc M.~Ismail, M.~E. Ismail, and W.~van Assche}, {\em Classical and Quantum
  Orthogonal Polynomials in One Variable}, vol.~13, Cambridge university press,
  2005.

\bibitem{kilbas2006theory}
{\sc A.~A. Kilbas, H.~M. Srivastava, and J.~J. Trujillo}, {\em Theory and
  Applications of Fractional Differential Equations}, vol.~204, Elsevier, 2006.

\bibitem{kobayashi2011stochastic}
{\sc K.~Kobayashi}, {\em Stochastic calculus for a time-changed semimartingale
  and the associated stochastic differential equations}, Journal of Theoretical
  Probability, 24 (2011), pp.~789--820.

\bibitem{kochubei2011general}
{\sc A.~N. Kochubei}, {\em General fractional calculus, evolution equations,
  and renewal processes}, Integral Equations and Operator Theory, 71 (2011),
  pp.~583--600.

\bibitem{kochubei2019growth}
{\sc A.~N. Kochubei and Y.~Kondratiev}, {\em Growth equation of the general
  fractional calculus}, Mathematics, 7 (2019), p.~615.

\bibitem{kolokol2019mixed}
{\sc V.~N. Kolokol’tsov}, {\em Mixed fractional differential equations and
  generalized operator-valued {M}ittag-{L}effler functions}, Mathematical
  Notes, 106 (2019), pp.~740--756.

\bibitem{kulik2013ergodicity}
{\sc A.~Kulik, N.~N. Leonenko, et~al.}, {\em Ergodicity and mixing bounds for
  the {F}isher--{S}nedecor diffusion}, Bernoulli, 19 (2013), pp.~2294--2329.

\bibitem{kumar2015inverse}
{\sc A.~Kumar and P.~Vellaisamy}, {\em Inverse tempered stable subordinators},
  Statistics \& Probability Letters, 103 (2015), pp.~134--141.

\bibitem{leonenko2013correlation}
{\sc N.~N. Leonenko, M.~M. Meerschaert, and A.~Sikorskii}, {\em Correlation
  structure of fractional {P}earson diffusions}, Computers \& Mathematics with
  Applications, 66 (2013), pp.~737--745.

\bibitem{leonenko2013fractional}
{\sc N.~N. Leonenko, M.~M. Meerschaert, and A.~Sikorskii}, {\em Fractional
  {P}earson diffusions}, Journal of Mathematical Analysis and Applications, 403
  (2013), pp.~532--546.

\bibitem{leonenko2017heavy}
{\sc N.~N. Leonenko, I.~Papi{\'c}, A.~Sikorskii, and N.~{\v{S}}uvak}, {\em
  Heavy-tailed fractional {P}earson diffusions}, Stochastic Processes and their
  Applications, 127 (2017), pp.~3512--3535.

\bibitem{leonenko2010statistical}
{\sc N.~N. Leonenko and N.~{\v{S}}uvak}, {\em Statistical inference for
  reciprocal gamma diffusion process}, Journal of Statistical Planning and
  Inference, 140 (2010), pp.~30--51.

\bibitem{leonenko2010statisticalb}
{\sc N.~N. Leonenko and N.~{\v{S}}uvak}, {\em Statistical inference for
  {S}tudent diffusion process}, Stochastic Analysis and Applications, 28
  (2010), pp.~972--1002.

\bibitem{linetsky2007spectral}
{\sc V.~Linetsky}, {\em Spectral methods in derivatives pricing}, Handbooks in
  Operations Research and Management Science, 15 (2007), pp.~223--299.

\bibitem{meerschaert2014tempered}
{\sc M.~M. Meerschaert, F.~Sabzikar, M.~S. Phanikumar, and A.~Zeleke}, {\em
  Tempered fractional time series model for turbulence in geophysical flows},
  Journal of Statistical Mechanics: Theory and Experiment, 2014 (2014),
  p.~P09023.

\bibitem{meerschaert2008triangular}
{\sc M.~M. Meerschaert and H.-P. Scheffler}, {\em Triangular array limits for
  continuous time random walks}, Stochastic Processes and their Applications,
  118 (2008), pp.~1606--1633.

\bibitem{meerschaert2011stochastic}
{\sc M.~M. Meerschaert and A.~Sikorskii}, {\em Stochastic Models for Fractional
  Calculus}, vol.~43, Walter de Gruyter, 2~ed., 2019.

\bibitem{meerschaert2013inverse}
{\sc M.~M. Meerschaert and P.~Straka}, {\em Inverse stable subordinators},
  Mathematical Modelling of Natural Phenomena, 8 (2013), pp.~1--16.

\bibitem{meerschaert2019relaxation}
{\sc M.~M. Meerschaert and B.~Toaldo}, {\em Relaxation patterns and
  semi-{M}arkov dynamics}, Stochastic Processes and their Applications, 129
  (2019), pp.~2850--2879.

\bibitem{metzler2000random}
{\sc R.~Metzler and J.~Klafter}, {\em The random walk's guide to anomalous
  diffusion: a fractional dynamics approach}, Physics Reports, 339 (2000),
  pp.~1--77.

\bibitem{oliveira2019anomalous}
{\sc F.~A. Oliveira, R.~Ferreira, L.~C. Lapas, and M.~H. Vainstein}, {\em
  Anomalous diffusion: A basic mechanism for the evolution of inhomogeneous
  systems}, Frontiers in Physics, 7 (2019), p.~18.

\bibitem{ozaki19852}
{\sc T.~Ozaki}, {\em Non-linear time series models and dynamical systems},
  Handbook of Statistics, 5 (1985), pp.~25--83.

\bibitem{patie2019spectral}
{\sc P.~Patie and A.~Srapionyan}, {\em Spectral projections correlation
  structure for short-to-long range dependent processes}, arXiv preprint
  arXiv:1905.10638,  (2019).

\bibitem{pazy2012semigroups}
{\sc A.~Pazy}, {\em Semigroups of Linear Operators and Applications to Partial
  Differential Equations}, vol.~44, Springer Science \& Business Media, 2012.

\bibitem{pearson1914tables}
{\sc K.~Pearson}, {\em Tables for Statisticians and Biometricians}, University
  press, 1914.

\bibitem{prosser1994kummer}
{\sc R.~T. Prosser}, {\em On the {K}ummer solutions of the hypergeometric
  equation}, The American Mathematical Monthly, 101 (1994), pp.~535--543.

\bibitem{revuz2013continuous}
{\sc D.~Revuz and M.~Yor}, {\em Continuous Martingales and Brownian Motion},
  vol.~293, Springer Science \& Business Media, 2013.

\bibitem{rudin1964principles}
{\sc W.~Rudin}, {\em Principles of Mathematical Analysis}, vol.~3, McGraw-hill,
  1964.

\bibitem{sansone1959orthogonal}
{\sc G.~Sansone}, {\em Orthogonal Functions}, Dover, 1991.

\bibitem{scalas2003revisiting}
{\sc E.~Scalas, R.~Gorenflo, F.~Mainardi, and M.~Raberto}, {\em Revisiting the
  derivation of the fractional diffusion equation}, Fractals, 11 (2003),
  pp.~281--289.

\bibitem{schilling2012bernstein}
{\sc R.~L. Schilling, R.~Song, and Z.~Vondracek}, {\em Bernstein Functions:
  Theory and Applications}, vol.~37, Walter de Gruyter, 2012.

\bibitem{schoutens2012stochastic}
{\sc W.~Schoutens}, {\em Stochastic Processes and Orthogonal Polynomials},
  vol.~146, Springer Science \& Business Media, 2012.

\bibitem{vsikic2006potential}
{\sc H.~{\v{S}}iki{\'c}, R.~Song, and Z.~Vondra{\v{c}}ek}, {\em Potential
  theory of geometric stable processes}, Probability Theory and Related Fields,
  135 (2006), pp.~547--575.

\bibitem{slater1966generalized}
{\sc L.~J. Slater}, {\em Generalized Hypergeometric Functions}, Cambridge
  University Press, 1966.

\bibitem{sokolov2005diffusion}
{\sc I.~M. Sokolov and J.~Klafter}, {\em From diffusion to anomalous diffusion:
  a century after einstein’s brownian motion}, Chaos: An Interdisciplinary
  Journal of Nonlinear Science, 15 (2005), p.~026103.

\bibitem{szeg1939orthogonal}
{\sc G.~Szego}, {\em Orthogonal Polynomials}, vol.~23, American Mathematical
  Soc., 1939.

\bibitem{toaldo2015convolution}
{\sc B.~Toaldo}, {\em Convolution-type derivatives, hitting-times of
  subordinators and time-changed ${C}_0$-semigroups}, Potential Analysis, 42
  (2015), pp.~115--140.

\bibitem{tomisaki1977asymptotic}
{\sc M.~Tomisaki}, {\em On the asymptotic behaviors of transition probability
  densities of one-dimensional diffusion processes}, Publications of the
  Research Institute for Mathematical Sciences, 12 (1977), pp.~819--834.

\bibitem{weidmann2006spectral}
{\sc J.~Weidmann}, {\em Spectral Theory of Ordinary Differential Operators},
  vol.~1258, Springer, 2006.

\bibitem{wong1964construction}
{\sc E.~Wong}, {\em The construction of a class of stationary {M}arkov
  processes}, Stochastic processes in mathematical physics and engineering, 17
  (1964), pp.~264--276.

\bibitem{zhang2012linking}
{\sc Y.~Zhang, M.~M. Meerschaert, and A.~I. Packman}, {\em Linking fluvial bed
  sediment transport across scales}, Geophysical Research Letters, 39 (2012).

\end{thebibliography}
\newpage
\appendix
\section{Proof of \cref{thmsemig}}\label{AppA}
To prove \cref{thmsemig} we will make use of the following representation of complete Bernstein functions (see \cite{schilling2012bernstein}, recall that we are supposing $a_\Phi=b_\Phi=0$):
\begin{equation*}
\Phi(\lambda)=\int_0^{+\infty}\frac{\lambda}{\lambda +t}\fs_\Phi(dt),
\end{equation*}
where $\int_0^{+\infty}\frac{1}{1+t}\fs_\Phi(dt)<+\infty$. The measure $\fs_\Phi(dt)$ is called the Stieltjes measure associated to $\Phi$. This representation leads to some interesting technical properties of the holomorphic extensions of complete Bernstein functions.
\begin{lemma}\label{lem:seccontr}
	Let $\Phi \not \equiv 0$ be a driftless complete Bernstein function. Then for any real number $\xi>0$ there exists $\alpha \in \left(0,\frac{\pi}{2}\right)$ and $K_i>0$, $i=1,2$, such that for any $\lambda \in \xi+\C\left(\frac{\pi}{2}+\alpha\right)$ it holds $\frac{|\Phi(\lambda)|}{|\Re(\Phi(\lambda))|}\le K_1$ and $\Re(\Phi(\lambda))\ge K_2>0$.
\end{lemma}
\begin{proof}
	First of all, let us observe that $\Phi$ admits an analytic continuation to the cut complex plane $\C \setminus (-\infty,0]$, as shown in \cite[Theorem $6.2$]{schilling2012bernstein}. Let us also recall that $\fs_\Phi(0,+\infty)=\lim_{\lambda \to +\infty}\Phi(\lambda)=+\infty$. Setting $\lambda=x+iy$, the real part of $\Phi(\lambda)$ is given by
	\begin{equation*}
	H_{\Re}(x,y)=\Re(\Phi(\lambda))=\int_0^{+\infty}\frac{x^2+y^2+tx}{(x+t)^2+y^2}\fs_\Phi(dt).
	\end{equation*}
	In particular, $\lim_{y \to +\infty}H(0,y)=1$, thus there exists $M_1>1$ such that $H(0,y)>\frac{1}{2}$ for any $y>M_1$. Defining, for $k>0$ and $y>0$, $F(k,y)=H_{\Re}(-ky,y)$, we obviously have $F(0,y)>\frac{1}{2}$ as $y>M_1$. We can also explicitly write $F(k,y)$ as
	\begin{equation*}
	F(k,y)=\int_0^{+\infty}\frac{(k^2+1)y^2}{(-ky+t)^2+y^2}\fs_\Phi(dt)-\int_0^{+\infty}\frac{tky}{(-ky+t)^2+y^2}\fs_\Phi(dt)
	\end{equation*}
	Now let us consider (setting $\min \emptyset=+\infty$)
	\begin{equation*}
	k(y)=\min\left\{k>0: \ F(k,y)=\frac{1}{2}\right\}, \ y>M_1;
	\end{equation*}
	we want to show that $\lim_{y \to +\infty}k(y)=+\infty$. Let us argue by contradiction, supposing that $\liminf_{y \to +\infty}k(y)=l<+\infty$. Then there exists a sequence $y_n \to +\infty$ such that $k(y_n)\to l$. In particular we have, by simple applications of dominated and monotone convergence theorem $\lim_{n \to +\infty}F(k(y_n),y_n)=+\infty$ that is absurd by definition of $k(y_n)$. Hence, for any $k_0>1$ there exists $M>M_1$ such that for any $y>M$ it holds $k(y)>k_0$. Let us fix $k_0>1$ and thus $M>M_1$.\\
	Now consider $r_0=M\sqrt{(k_0^2+1)}$. Expressing $H_{\Re}$ in polar coordinates we can define
	\begin{equation*}
	G(r,\theta)=\int_0^{+\infty}\frac{r^2+tr\cos(\theta)}{(r\cos(\theta)+t)^2+r^2\sin^2(\theta)}\fs_\Phi(dt)
	\end{equation*}
	for $r \in [1,r_0]$ and $\theta \in (0,\pi)$. Moreover, denote	$G_m(\theta)=\min_{r \in [1,r_0]}G(r,\theta)$ that is continuous by Berge's theorem and $G_m\left(\frac{\pi}{2}\right)=C_1>0$; for this reason there exists $\alpha'>0$ such that if $\theta \in \left[\frac{\pi}{2},\frac{\pi}{2}+\alpha'\right]$ then $G_m(\theta)>\frac{C_1}{2}$ and, without loss of generality, we can choose $\alpha'$ to be small enough to have $\tan\left(\frac{\pi}{2}+\alpha'\right)<-\frac{1}{k_0}$. Define $C_2=\frac{1}{2}\min\{C_1,1\}$ and the set
	\begin{equation*}
	A=\left\{z=re^{i\theta}, \ r \ge 1, \ \theta \in \left[\frac{\pi}{2},\frac{\pi}{2}+\alpha'\right]\right\}
	\end{equation*}
	in such a way that for any $\lambda \in A$ it holds $\Re(\Phi(\lambda))>C_2>0$. Moreover, observing that
	\begin{equation*}
	H_{\Im}(x,y):=\Im(\Phi(\lambda))=\int_0^{+\infty}\frac{yt}{(x+t)^2+y^2}s_\Phi(dt)
	\end{equation*}
	it is easy to check that for $y>0$ it holds $H_{\Im}(x,y)>0$ and that, for any fixed $x \in \R$, $\lim_{y \to +\infty}H_{\Im}(x,y)=0$. Thus, arguing as we did for the real part, there exists a constant $H_1>0$ such that for $\lambda \in A$ it holds $\Im(\Phi(\lambda))\le H_1 \Re(\Phi(\lambda))$.\\
	Now set $m'=\tan\left(\frac{\pi}{2}+\alpha'\right)$. For $\xi \le -\frac{1}{m'}$ define $m=-\frac{1}{\xi}$ and $\alpha>0$ in such a way that $m=\tan\left(\frac{\pi}{2}+\alpha\right)$. If $\xi>-\frac{1}{m}$, just set $\alpha=\alpha'$ and $m=m'$.\\ 
	Now let us consider any $\lambda \in \xi +\C\left(\frac{\pi}{2}+\alpha\right)$ and write $\lambda=x+iy$. If $x \le 0$ and $y>0$ then, by definition, $\lambda \in A$ and then $\Re(\Phi(\lambda))>C_2$. If $x>0$, then $y>m(x-\xi)$ and in particular $x>\frac{m\xi+y}{m}$, that is positive as $y \in (0,1)$. Now, by using the L\'evy-Khintchine representation of $\Phi$, we have
	\begin{equation*}
	\Re(\Phi(\lambda))=\int_0^{+\infty}(1-e^{-xt}\cos(yt))\nu_\Phi(dt).
	\end{equation*}
	Denote by $S$ the segment with extrema $\xi$ and $-im\xi$ and $C_3=\min_{\lambda \in S}\Re(\Phi(\lambda))$. Now, let us observe that for fixed $y \in (0,1)$ the function $x \in \left(\frac{m\xi+y}{m},+\infty\right) \mapsto \Re(\Phi(\lambda))$ is increasing and thus $\Re(\Phi(\lambda))\ge C_3$ if $y \in (0,1)$. If $y \ge 1$, the function $x \in \left(0,+\infty\right) \mapsto \Re(\Phi(\lambda))$ is increasing and then $\Re(\Phi(\lambda))\ge C_2$ by definition of $A$. On the other hand, observing that
	\begin{equation*}
	\Im(\Phi(\lambda))=\int_0^{+\infty}e^{-xt}\sin(yt)\nu_\Phi(dt),
	\end{equation*}
	for fixed $y \in (0,1)$ the function $x \in \left(\frac{m\xi+y}{m},+\infty\right) \mapsto \Im(\Phi(\lambda))$ is decreasing and the same holds for fixed $y \ge 1$ considering $x \in \left(0,+\infty\right) \mapsto \Im(\Phi(\lambda))$. Thus, there exists a constant $H_2$ such that for any $\lambda \in \xi+\C\left(\frac{\pi}{2}+\alpha\right)$ with $x>0$ and $y \ge 0$ it holds $\Im(\Phi(\lambda))\le H_2\Re(\Phi(\lambda))$. Setting $K_2=\min\{C_2,C_3\}$ and $H=\max\{H_1,H_2\}$, that we have shown that for any $\lambda \in \xi+\C\left(\frac{\pi}{2}+\alpha\right)$ with $y \ge 0$ it holds $\Re(\Phi(\lambda))\ge K_2$ and $\Im(\Phi(\lambda))\le H\Re(\Phi(\lambda))$. Concerning $y<0$, let us recall that, since for any $\lambda \in \R^+$ it holds $\Phi(\lambda) \in \R^+$, then for any $\lambda \in \xi+\C\left(\frac{\pi}{2}+\alpha\right)$ with $y < 0$ it holds $\overline{\Phi(\lambda)}=\Phi(\overline{\lambda})$ and then $\Re(\Phi(\lambda))=\Re(\Phi(\overline{\lambda})) \ge K_2$ and $\Im(\Phi(\lambda))=-\Im(\Phi(\overline{\lambda}))\ge -H\Re(\Phi(\overline{\lambda}))=-H\Re(\Phi(\lambda))$. In particular we have also shown in this way that there exists $\beta<\frac{\pi}{2}$ such that $\Phi\left(\xi+\C\left(\frac{\pi}{2}+\alpha\right)\right)\subseteq \C\left(\beta\right)$.\\
	Let us observe that if we consider $\lambda \in \xi+\C\left(\frac{\pi}{2}+\alpha\right)$ we have that $\Phi(\lambda)=re^{i\theta}$ for some $\theta \in (-\beta,\beta)$. Since $\beta<\frac{\pi}{2}$ we get
	\begin{equation*}
	\frac{|\Phi(\lambda)|}{|\Re(\Phi(\lambda))|} =\frac{1}{\cos(\theta)}\ge \frac{1}{\cos(\beta)}=:K_1
	\end{equation*}
	concluding the proof.
\end{proof}
Now we are ready to prove \cref{thmsemig}.
\begin{proof}[Proof of \cref{thmsemig}]
	Let us first show that the family of operators $(T_\Phi(t))_{t \ge 0}$ is well-defined and uniformly bounded. To do this, we need to use Bochner's theorem. Indeed we have, for any $u \in X$,
	\begin{equation*}
	\Norm{T_\Phi(t)u}{}\le \int_0^{+\infty}\Norm{T(s)u}{}f_\Phi(s,t)ds\le M\Norm{u}.
	\end{equation*}
	Now let us show strong continuity of the family $(T_\Phi(t))_{t \ge 0}$ at $0^+$. Indeed we have that $t \in \R^+_0 \mapsto T(t)u \in X$ is a bounded continuous function (by strong continuity and uniform boundedness of the semi-group $T(t)$) and so it is $t \in \R^+_0 \mapsto \Norm{T(t)u}{}\in \R^+_0$. Thus it holds
	\begin{equation*}
	\Norm{T_\Phi(t)u-u}{}\le \int_0^{+\infty}\Norm{T(s)u-u}{}f_\Phi(s;t)ds.
	\end{equation*}
	Taking the limit we obtain,
	\begin{equation*}
	\lim_{t \to 0^+}\Norm{T_\Phi(t)u-u}{}\le \Norm{T(0)u-u}{}=0,
	\end{equation*}
	since, by dominated convergence theorem, $f_\Phi(s;t)$ weakly converges towards $\delta_0$ as $t \to 0^+$. Thus we have strong continuity at $0^+$. Strong continuity at any $t>0$ can be analogously proven.\\
	To show strong analyticity of $(T_\Phi(t))_{t \ge 0}$ in a certain sectorial region, we will argue by using Laplace transforms. First of all, let us fix $\lambda>0$ and consider, for $u \in X$,
	\begin{align*}
	\int_0^{+\infty}e^{-s\lambda}T_\Phi(s)uds&=\int_0^{+\infty}e^{-s\lambda}\int_0^{+\infty}T(\tau)uf_\Phi(\tau,s)d\tau ds\\
	&=\frac{\Phi(\lambda)}{\lambda}\int_0^{+\infty}e^{-\tau\Phi(\lambda)}T(\tau)ud\tau,
	\end{align*}
	where we could use Fubini's theorem since
	\begin{equation*}
	\int_0^{+\infty}e^{-\tau\Phi(\lambda)}\Norm{T(\tau)u}{}d\tau\le \frac{M \Norm{u}{}}{\Phi(\lambda)}<+\infty.
	\end{equation*}
	Denoting by $r(\lambda)=\int_0^{+\infty}e^{-s\lambda}T_\Phi(s)uds$ and by $q(\lambda)=\int_0^{+\infty}e^{-\tau \lambda}T(\tau)ud\tau$, we have
	\begin{equation}\label{equal}
	\lambda r(\lambda)=\Phi(\lambda)q(\Phi(\lambda)).
	\end{equation}
	Let us extend $q(\lambda)$ to the whole semi-plane $\bH=\{\lambda \in \C: \ \Re(\lambda)>0\}$. This can be done since for $\lambda \in \bH$, by Bochner's theorem,
	\begin{equation*}
	\Norm{q(\lambda)}{}\le \int_0^{+\infty}e^{-\tau \Re(\lambda)}\Norm{T(\tau)u}{}d\tau \le \frac{M \Norm{u}{}}{\Re(\lambda)}.
	\end{equation*}
	Now we have to consider an analytic continuation of $r(\lambda)$ to a suitable sector. To do this, let us consider the fact that $r(\lambda)=\frac{\Phi(\lambda)}{\lambda}q(\Phi(\lambda))$ on the whole real line. By identity of analytic functions (see, for instance, \cite[Theorem $3.11.5$]{hille1996functional}), we only have to extend the right hand side of the equality. Thus, let us consider any $\xi>0$ and the sector $\xi+\C\left(\frac{\pi}{2}+\alpha\right)$ as defined in \cref{lem:seccontr}. We have that $\frac{\Phi(\lambda)}{\lambda}$ is defined and analytic on the sector $\xi+\C\left(\frac{\pi}{2}+\alpha\right)$ (see \cite[Theorem $6.2$]{schilling2012bernstein}). Moreover, for $\lambda \in \xi+\C\left(\frac{\pi}{2}+\alpha\right)$ we have $\Re(\Phi(\lambda))\ge K_2$ for some constant $K_2>0$, hence $q(\Phi(\lambda))$ is well defined and analytic on $\xi+\C\left(\frac{\pi}{2}+\alpha\right)$. Thus we can conclude that $r(\lambda)$ is well defined and analytic on $\xi+\C\left(\frac{\pi}{2}+\alpha\right)$.\\
	Now let us show that the quantity $\Norm{(\lambda-\xi)r(\lambda)}{}$ is bounded in  $\xi+\C\left(\frac{\pi}{2}+\alpha\right)$. Indeed we have
	\begin{align*}
	\Norm{\lambda r(\lambda)}{}\le  |\Phi(\lambda)|\Norm{q(\Phi(\lambda))}{} \le   \frac{|\Phi(\lambda)|}{\Re(\Phi(\lambda))}M\Norm{u}{}\le K_1M\Norm{u}{}.
	\end{align*}
	On the other hand, we have
	\begin{equation*}
	\Norm{\xi r(\lambda)}{}\le \frac{|\Phi(\lambda)|\xi}{|\lambda|}\Norm{q(\Phi(\lambda))}{}\le \frac{K_1M\Norm{u}{}\xi}{K_3},
	\end{equation*}
	where $K_3$ is defined in such a way that $|\lambda|<K_3$ implies $\lambda \not \in \xi+\C\left(\frac{\pi}{2}+\alpha\right)$. Thus, finally, we achieve
	\begin{equation*}
	\Norm{(\lambda-\xi) r(\lambda)}{}\le \left(1+\frac{\xi}{K_3}\right)K_1M\Norm{u}{}.
	\end{equation*}
	Then, \cite[Theorem $2.6.1$]{arendtvector} implies that $r(\lambda)$ is the Laplace transform of some analytic function in $\C\left(\alpha\right)$. This, in particular, tells us that $T_\Phi(t)u$ admits an analytic extension to the whole sector $\C(\alpha)$.\\
	Now let us consider the function
	\begin{equation*}
	\widetilde{T}_\Phi(t)u=\int_0^t \bar{\nu}_\Phi(t-s)T_\Phi(s)uds.
	\end{equation*} 
	We want to show that such function admits an analytic extension up to the sector $\C(\alpha)$. To do this, let us argue again via Laplace transform. Indeed, taking the Laplace transform of $\widetilde{T}_\Phi(t)u$ for $\lambda>0$ we obtain
	\begin{equation*}
	\widetilde{r}(\lambda)=\int_0^{+\infty}e^{-t\lambda}\widetilde{T}_\Phi(t)udt=\frac{\Phi(\lambda)}{\lambda}r(\lambda),
	\end{equation*} 
	that, still by identity of analytic functions, is well defined for any $\lambda \in \xi+\C\left(\frac{\pi}{2}+\alpha\right)$. Again, we have
	\begin{equation*}
	\Norm{(\lambda-x)\widetilde{r}(\lambda)}{}= \frac{|\Phi(\lambda)|}{|\lambda|}\Norm{(\lambda-x)r(\lambda)}{}\le \left(1+\frac{\xi}{K_3}\right)K_1M\Norm{u}{}\frac{|\Phi(\lambda)|}{|\lambda|}.
	\end{equation*}
	Let us show that for $\lambda \in \xi+\C\left(\frac{\pi}{2}+\alpha\right)$ the fraction $\frac{|\Phi(\lambda)|}{|\lambda|}$ is bounded. Indeed, first of all, $\lambda$ is separated from $0$, since $|\lambda| \ge K_3$. Moreover, $\frac{\Phi(\lambda)}{\lambda}$ is continuous on $\xi+\C\left(\frac{\pi}{2}+\alpha\right)$, we only have to show that it is bounded for big values of $|\lambda|$. However, by \cite[Corollary $6.5$]{schilling2012bernstein} (precisely the last observation in the proof of the aforementioned corollary), we have, setting $\lambda=Re^{i \theta}$,
	\begin{equation*}
	\lim_{R \to +\infty}\frac{\Phi(Re^{i\theta})}{Re^{i \theta}}=0
	\end{equation*}
	uniformly with respect to $\theta$ such that $\lambda \in \C\left(\frac{\pi}{2}+\alpha\right)$. Observing that $\xi+\C\left(\frac{\pi}{2}+\alpha\right)\subset \C\left(\frac{\pi}{2}+\alpha\right)$, we get that $\frac{|\Phi(\lambda)|}{|\lambda|}$ is bounded in $\xi+\C\left(\frac{\pi}{2}+\alpha\right)$. In particular, still by \cite[Theorem $2.6.1$]{arendtvector}, this proves that $\widetilde{r}(\lambda)$ is the Laplace transform of some analytic function in $\C(\alpha)$, and then that $\widetilde{T}_\Phi(t)u$ admits an analytic extension in $\C(\alpha)$.\\
	Now let us consider the generator $(A,\cD(A))$ of the semi-group $T(t)$ and $u \in \cD(A)$. Then we know that $\int_0^t T(s)uds \in \cD(A)$, $T(t)u \in \cD(A)$ and
	\begin{equation}\label{eq:prec}
	T(t)u=A\int_0^t T(s)uds+u,
	\end{equation}
	as shown in \cite[Theorem $1.2.4.b$]{pazy2012semigroups}. Now let us consider the Laplace transform of $\int_0^t T(s)uds$, that is given by $\frac{1}{\lambda}q(\lambda)$. For $\lambda>0$ we have that $q(\lambda)$ can be expressed in terms of a Riemann sum. In particular we can consider $q_n \to q(\lambda)$ where $q_n$ are finite sums. Moreover, define $h(\lambda)=\frac{1}{\lambda}q(\lambda)$ and $h_n=\frac{1}{\lambda}q_n$ to achieve $h_n \to h(\lambda)$. Since $q_n$ are finite sums, we have $q_n \in \cD(A)$ and then $h_n \in \cD(A)$ with $Ah_n=\frac{1}{\lambda}Aq_n$. By \cite[Corollary $1.2.5$]{pazy2012semigroups}, we know that $A$ is a closed operator, thus $h(\lambda) \in \cD(A)$ and $Ah_n \to Ah(\lambda)$. For this reason also $q(\lambda) \in \cD(A)$ and $Aq_n \to Aq(\lambda)$. Finally, we have
	\begin{equation*}
	Ah(\lambda)=\lim_{n \to +\infty}Ah_n=\frac{1}{\lambda}\lim_{n \to +\infty}Aq_n=\frac{1}{\lambda}Aq(\lambda).
	\end{equation*}
	Hence we can take the Laplace transform on both sides of \cref{eq:prec} and multiply both sides of the equations by $\lambda>0$ to achieve
	\begin{equation*}
	\lambda q(\lambda)=Aq(\lambda)+u.
	\end{equation*}
	Now let us observe that being $\Phi$ a Bernstein function, it is positive in $\R^+$ and then we can substitute $\lambda$ with $\Phi(\lambda)$ and use relation \cref{equal}, obtaining
	\begin{equation*}
	\lambda r(\lambda)=\frac{\lambda}{\Phi(\lambda)}Ar(\lambda)+u.
	\end{equation*}
	Multiplying everything by $\frac{\Phi(\lambda)}{\lambda^2}$ and taking the term $\frac{\Phi(\lambda)}{\lambda}u$ on the other side, we conclude
	\begin{equation}\label{Laptranseq}
	\frac{\Phi(\lambda)}{\lambda} r(\lambda)-\frac{\Phi(\lambda)}{\lambda^2}u=A\frac{1}{\lambda}r(\lambda).
	\end{equation}
	Now let us observe that we have shown that the left hand side is the Laplace transform of some function, precisely of $\widetilde{T}_\Phi(t)u-u\cI_\Phi(t)$, where $\cI_\Phi(t)=\int_0^t\bar{\nu}_\Phi(s)ds$. Moreover, let us observe that for $\lambda \in \xi + \C\left(\frac{\pi}{2}+\alpha\right)$ it holds
	\begin{equation*}
	\Norm{(\lambda-\xi)\frac{\Phi(\lambda)}{\lambda^2}u}{}=\frac{|\lambda-\xi|}{|\lambda|}\frac{|\Phi(\lambda)|}{|\lambda|}\Norm{u}{}.
	\end{equation*}
	We have already shown that $\frac{|\Phi(\lambda)|}{|\lambda|}$ is bounded in $\xi + \C\left(\frac{\pi}{2}+\alpha\right)$. Moreover, if $\lambda=\xi+re^{i\theta}$, we have $|\lambda-\xi|=r$ and 
	\begin{equation*}
	\lambda=(\xi+r\cos(\theta))^2+r^2\sin^2(\theta)=\xi^2+r^2+2\xi r\cos(\theta).
	\end{equation*}
	Thus we have
	\begin{equation*}
	\lim_{r \to +\infty}\frac{|\lambda-\xi|}{|\lambda|}=\lim_{r \to +\infty}\frac{r}{\sqrt{\xi^2+r^2+2\xi r\cos(\theta)}}=1.
	\end{equation*}
	This, together with the fact that $\lambda$ is separated from $0$, tells us that $\frac{|\lambda-\xi|}{|\lambda|}$ is bounded on $\xi +\C\left(\frac{\pi}{2}+\alpha\right)$ and then $u\cI_\Phi(t)$ can be extended with an analytical function up to the sector $\C(\alpha)$. Remind that also $\widetilde{T}_\Phi(t)u$ can be extended with an analytical function up to the sector $\C(\alpha)$.\\
	Thus, also the right-hand side of \cref{Laptranseq} is the Laplace transform of some analytical function and we can consider the inverse Laplace transform. To invert this, let us use \cite[Corollary $1.4$]{baeumer2003inversion}. Hence we have
	\begin{equation}\label{eqpre}
	\widetilde{T}_\Phi(t)u-u\cI_\phi(t)=\lim_{n \to +\infty}\sum_{j=1}^{N_n}\alpha_{j,n}e^{\beta_{j,n}t}A\frac{1}{\beta_{j,n}}r(\beta_{j,n})
	\end{equation}
	where $N_n$, $\alpha_{j,n}$ and $\beta_{j,n}$ are defined in the Corollary and the limit in the right hand side is uniform on compact sets. Let us observe that, for fixed $\lambda$, identity \cref{Laptranseq} implies that $\frac{r(\lambda)}{\lambda}\in \cD(A)$ and thus also $r(\lambda) \in \cD(A)$. Since $A$ is linear we obtain
	\begin{equation*}
	\sum_{j=1}^{N_n}\alpha_{j,n}e^{\beta_{j,n}t}A\frac{1}{\beta_{j,n}}r(\beta_{j,n})=A\sum_{j=1}^{N_n}\alpha_{j,n}e^{\beta_{j,n}t}\frac{1}{\beta_{j,n}}r(\beta_{j,n}).
	\end{equation*}
	However, we have that $\frac{r(\lambda)}{\lambda}$ is the Laplace transform of $\int_0^t T_\Phi(s)uds$, that is analytic in $\C(\alpha)$ since it is the integral of an analytic function. Hence, still by \cite[Corollary $1.4$]{baeumer2003inversion}, we have
	\begin{equation*}
	\lim_{n \to +\infty}\sum_{j=1}^{N_n}\alpha_{j,n}e^{\beta_{j,n}t}\frac{1}{\beta_{j,n}}r(\beta_{j,n})=\int_0^t T_\Phi(s)uds.
	\end{equation*}
	Finally, being $A$ a closed operator we get
	\begin{equation*}
	\lim_{n \to +\infty}\sum_{j=1}^{N_n}\alpha_{j,n}e^{\beta_{j,n}t}A\frac{1}{\beta_{j,n}}r(\beta_{j,n})=A\int_0^t T_\Phi(s)uds.
	\end{equation*}
	We have from Equation \cref{eqpre}
	\begin{equation}\label{eqpre2}
	\widetilde{T}_\Phi(t)u-u\cI_\Phi(t)=A\int_0^t T_\Phi(s)uds.
	\end{equation}
	The left-hand side is analytic in $\C(\alpha)$, thus we can derive it with respect to $t>0$. The derivative of the left-hand side admits the Laplace transform (see \cite[Corolary $1.6.6$]{arendtvector} and observe that $\widetilde{T}_\Phi(0)u-u\cI_\Phi(0)=0$)
	\begin{equation*}
	\Phi(\lambda)r(\lambda)-\frac{\Phi(\lambda)}{\lambda}u
	\end{equation*}
	which is actually the Laplace transform of $\partial_t^\Phi T_\Phi(t)u$. Hence we obtain, by uniqueness of the Laplace transform and the fact that $A$ is a closed operator, taking the derivative on both sides of Equation \cref{eqpre2},
	\begin{equation*}
	\partial_t^\Phi T_\Phi(t)u=AT_\Phi(t)u,
	\end{equation*}
	concluding the proof.
\end{proof}
\end{document}